\UseRawInputEncoding
\documentclass[12pt]{article}

\usepackage{dsfont}
\usepackage{amsfonts,amsmath,amsthm}
\usepackage{amssymb}
\usepackage{algorithm}
\usepackage{algpseudocode}
\usepackage{xcolor}
\usepackage{graphicx}
\usepackage{stmaryrd}        
\usepackage{multirow}

\usepackage[paper=a4paper,dvips,top=2cm,left=2cm,right=2cm,
foot=1cm,bottom=4cm]{geometry}

\newcommand{\ve}{{\bf e}}

\begin{document}
	
	\title{Biquadratic Tensors: Eigenvalues and Structured Tensors}
	\author{Liqun Qi\footnote{Department of Applied Mathematics, The Hong Kong Polytechnic University, Hung Hom, Kowloon, Hong Kong.
			({\tt maqilq@polyu.edu.hk}).}
		\and { \
			Chunfeng Cui\footnote{LMIB of the Ministry of Education, School of Mathematical Sciences, Beihang University, Beijing 100191 China.
				({\tt chunfengcui@buaa.edu.cn}).}
		}
	}
	\date{\today}
	\maketitle
	
	\begin{abstract}
		The covariance tensors in statistics{, elasticity tensor in solid {mechanics}, Riemann curvature tensor in relativity theory} are {all} biquadratic tensors that are weakly symmetric, but not symmetric in general.
		Motivated by this, in this paper, we consider nonsymmetric biquadratic tensors, and study possible conditions and algorithms for identifying positive semi-definiteness and definiteness of such biquadratic tensors.   We extend M-eigenvalues to nonsymmetric biquadratic tensors, prove that a general biquadratic tensor has at least one M-eigenvalue, and show that a general biquadratic tensor is positive semi-definite if and only if all of its M-eigenvalues are nonnegative, and a general biquadratic tensor is positive definite if and only if all of its M-eigenvalues are positive.
		We present a Gershgorin-type theorem for biquadratic tensors, and show that (strictly) diagonally  {dominated  biquadratic} tensors are positive semi-definite (definite).      We introduce {Z-biquadratic tensors,}  M-biquadratic tensors, strong M-biquadratic tensors, {B$_0$}-biquadratic tensors and B-biquadratic tensors.  We show that  {M-biquadratic} tensors and symmetric B$_0$-biquadratic tensors are positive semi-definite, and that {strong} M-biquadratic tensors and symmetric B-biquadratic tensors are positive definite.    A {Riemannian} LBFGS method  for computing the smallest M-eigenvalue of a general biquadratic tensor is presented.  Numerical results are reported.

		\medskip


		\textbf{Key words.} Biquadratic tensors, M-eigenvalues, positive semi-definiteness, Gershgorin-type theorem, diagonally dominated biquadratic tensors, M-biquadratic tensors, B-biquadratic tensors.
		
		\medskip
		\textbf{AMS subject classifications.} 47J10, 15A18, 47H07, 15A72.
	\end{abstract}

	\renewcommand{\Re}{\mathds{R}}
	\newcommand{\rank}{\mathrm{rank}}
	\newcommand{\X}{\mathcal{X}}
	\newcommand{\A}{\mathcal{A}}
	\newcommand{\I}{\mathcal{I}}
	\newcommand{\B}{\mathcal{B}}
	\newcommand{\C}{\mathcal{C}}
	\newcommand{\D}{\mathcal{D}}
	\newcommand{\LL}{\mathcal{L}}
	\newcommand{\OO}{\mathcal{O}}
	\newcommand{\e}{\mathbf{e}}
	\newcommand{\0}{\mathbf{0}}
	\newcommand{\1}{\mathbf{1}}
	\newcommand{\dd}{\mathbf{d}}
	\newcommand{\ii}{\mathbf{i}}
	\newcommand{\jj}{\mathbf{j}}
	\newcommand{\kk}{\mathbf{k}}
	\newcommand{\va}{\mathbf{a}}
	\newcommand{\vb}{\mathbf{b}}
	\newcommand{\vc}{\mathbf{c}}
	\newcommand{\vq}{\mathbf{q}}
	\newcommand{\vg}{\mathbf{g}}
	\newcommand{\pr}{\vec{r}}
	\newcommand{\pc}{\vec{c}}
	\newcommand{\ps}{\vec{s}}
	\newcommand{\pt}{\vec{t}}
	\newcommand{\pu}{\vec{u}}
	\newcommand{\pv}{\vec{v}}
	\newcommand{\pn}{\vec{n}}
	\newcommand{\pp}{\vec{p}}
	\newcommand{\pq}{\vec{q}}
	\newcommand{\pl}{\vec{l}}
	\newcommand{\vt}{\rm{vec}}
	\newcommand{\x}{\mathbf{x}}
	\newcommand{\vx}{\mathbf{x}}
	\newcommand{\vy}{\mathbf{y}}
	\newcommand{\vu}{\mathbf{u}}
	\newcommand{\vv}{\mathbf{v}}
	\newcommand{\y}{\mathbf{y}}
	\newcommand{\vz}{\mathbf{z}}
	\newcommand{\vp}{\mathbf{p}}
	\newcommand{\T}{\top}
	\newcommand{\R}{\mathcal{R}}
	
	\newtheorem{Thm}{Theorem}[section]
	\newtheorem{Def}[Thm]{Definition}
	\newtheorem{Ass}[Thm]{Assumption}
	\newtheorem{Lem}[Thm]{Lemma}
	\newtheorem{Prop}[Thm]{Proposition}
	\newtheorem{Cor}[Thm]{Corollary}
	\newtheorem{example}[Thm]{Example}
	\newtheorem{remark}[Thm]{Remark}
	
	\section{Introduction}

	Let $m, n$ be positive integers and $m, n \ge 2$.   We call a real fourth order  $(m \times n \times m \times n)$-dimensional tensor $\A = (a_{i_1j_1i_2j_2}) \in \Re^{m \times n \times m \times n}$ a {\bf biquadratic tensor}.  This is different from  \cite{QHZX21}. Denote $[n] := \{ 1, {\dots,} n \}$.  If for all $i_1, i_2 \in [m]$ and $j_1, j_2\in [n]$, we have
	$$a_{i_1j_1i_2j_2} = a_{i_2j_2i_1j_1},$$
	then we say that $\A$ is {weakly symmetric.}   If furthermore for all $i_1, i_2 \in [m]$ {and $j_1, j_2\in [n]$, we have}
	$$a_{i_1j_1i_2j_2} = a_{i_2j_1i_1j_2}{=a_{i_1j_2i_2j_1}},$$
	then we say $\A$ is {symmetric.}
	{In fact, {all} the biquadratic tensors in \cite{QHZX21} are   symmetric biquadratic tensors in this paper.}
	{Denote} the set of all biquadratic tensors in $\Re^{m \times n \times m \times n}$ by $BQ(m, n)$,  the set of all {weakly} symmetric biquadratic tensors in $BQ(m, n)$, by {$WBQ(m, n)$,} and the set of all symmetric biquadratic tensors in $BQ(m, n)$, by $SBQ(m, n)$, {respectively}.
	Then $BQ(m, n)$, $WBQ(m, n)$, {and} $SBQ(m, n)$ are {all} linear spaces.

	A {biquadratic} tensor {$\A\in BQ(m,n)$} is called {\bf positive semi-definite} if for any $\x \in \Re^m$ and $\y \in \Re^n$,
	\begin{equation}\label{equ:PSD}
		{f(\x, \y) \equiv} \langle \A, \x \circ \y \circ \x \circ \y \rangle \equiv \sum_{i_1, i_2 =1}^m \sum_{j_1, j_2 = 1}^n a_{i_1j_1i_2j_2}x_{i_1}y_{j_1}x_{i_2}y_{j_2} \ge 0.
	\end{equation}
	The tensor $\A$ is called an SOS (sum-of-squares) biquadratic tensor if $f(\x,\y)$ can be written as a sum of squares.
	A {biquadratic} tensor {$\A\in BQ(m,n)$} is called {\bf positive definite} if for any $\x \in \Re^m, \x^\top \x = 1$ and $\y \in \Re^n, \y^\top \y = 1$,
	$${f(\x, \y) \equiv} \langle \A, \x \circ \y \circ \x \circ \y \rangle \equiv \sum_{i_1, i_2 =1}^m \sum_{j_1, j_2 = 1}^n a_{i_1j_1i_2j_2}x_{i_1}y_{j_1}x_{i_2}y_{j_2} > 0.$$
	
	In \cite{HLW21}, bi-block tensors were studied.  A bi-block tensor may have order higher than $4$, but its dimension is uniformly $n$. Recall that a tensor $\A = (a_{i_1\cdots i_m})$ with $i_1, \cdots, i_m \in [n]$ is called a cubic tensor \cite{KB09}.    Hence, bi-block tensors are special  cubic tensors.   On the other hand,  biquadratic tensors are not cubic tensors.  Thus, they are different.

	Biquadratic tensors arise from solid mechanics, statistics, quantum physics, spectral graph theory and polynomial theory. In the next section, we review such application background of biquadratic tensors.  In particular, we point out that
	the covariance tensors in statistics{, elasticity tensor in solid {mechanics}, Riemann curvature tensor in relativity theory} are {all} biquadratic tensors that are weakly symmetric, but not symmetric in general.
	This motivates us to consider nonsymmetric biquadratic tensors, and study possible conditions and algorithms for identifying positive semi-definiteness and definiteness of such biquadratic tensors.
	
	In the study of strong ellipticity condition of the elasticity tensor in solid mechanics, in 2009, Qi, Dai and Han \cite{QDH09} introduced M-eigenvalues for symmetric biquadratic tensors.
	In Section 3, we extend this definition to nonsymmetric biquadratic tensors.   We show that
	for any biquadratic tensor, an M-eigenvalue defined in this paper always exists.  Furthermore, we show that a general biquadratic tensor is positive semi-definite if and only if all of its M-eigenvalues are nonnegative, and a general biquadratic tensor is positive definite if and only if all of its M-eigenvalues are positive.  This creates a powerful tool for identifying positive semi-definiteness and definiteness of general biquadratic tensors.
	
	In Section 4, we present a Gershgorin-type theorem for biquadratic tensors.  We introduce diagonally dominated biquadratic tensors and strictly diagonally dominated biquadratic tensors.  Then we show that diagonally  {dominated biquadratic} tensors are positive semi-definite, and strictly diagonally  {dominated biquadratic} tensors are positive definite.
	
	
	In  {Section 5}, we introduce {Z-biquadratic tensors,} M-biquadratic tensors, strong M-biquadratic tensors, {B$_0$}-biquadratic tensors and B-biquadratic tensors.  We show {that M-biquadratic} tensors and symmetric B$_0$-biquadratic tensors are positive semi-definite, and  {that strong} M-biquadratic tensors and symmetric B-biquadratic tensors are positive definite.  
	A {Riemannian} LBFGS method
	{for computing the smallest M-eigenvalue of a general biquadratic tensor is presented in Section 6.}  Convergence analysis of this algorithm is also given there. Numerical results are reported in Section 7.   {Some final remarks are made in Section 8.}
	
	

	\section{Application Backgrounds of Biquadratic Tensors}

	\subsection{The Covariance Tensor in Statistics}

	The covariance matrix is a key concept in statistics and machine learning. It is a square matrix that contains the covariances between multiple random variables. The diagonal elements represent the variances of each variable, reflecting their individual dispersion. The off-diagonal elements represent the covariances between pairs of variables, indicating their linear relationships. The covariance matrix is symmetric and positive semi-definite, providing a comprehensive view of the interdependencies among variables.
	However, when the variable takes the form of a matrix, the corresponding covariance matrix transforms into a covariance tensor.

	Let $X=(X_{ij})\in\Re^{{m\times n}}$ be a random matrix  with each element being a random variable.
	We denote the mean and variance of the element $X_{ij}$ as $\mathbb E(X_{ij})=\mu_{ij}$ and $\text{var}(X_{ij})=\sigma_{ij}$ (${i\in [m], j\in [n]}$), respectively.  The covariance with another element $X_{kl}$ is denoted as $\text{cov}(X_{ij},X_{kl})=\sigma_{ijkl}$. {Subsequently, it is formulated as a fourth-order covariance tensor.}
	The   {fourth order} covariance tensor  {was} proposed in \cite{BLQZ12} for portfolio selection problems and for  group
	identification in \cite{CHHS25}, respectively.
	
	For any random matrix $X\in\Re^{{m\times n}}$, its fourth-order covariance tensor is defined as $\A=(\sigma_{i_1j_1i_2j_2}) \in\Re^{{m\times n}\times {m\times n}}$, where
	\begin{equation}\label{eq:4cov_ten}
		\sigma_{i_1j_1i_2j_2} = \mathbb  E\left((X_{i_1j_1}-\mu_{i_1j_1})(X_{i_2j_2}-\mu_{i_2j_2})\right).
	\end{equation}
	Then  $\A$ is a {weakly symmetric} biquadratic tensor.  $\A$ may   be nonsymmetric.

	\begin{Prop}\label{covtensor:PSD}
		The fourth-order covariance tensor defined in  \eqref{eq:4cov_ten} is   positive semi-definite.
	\end{Prop}
	\begin{proof}
		For any $\x \in \Re^m$ and $\y \in \Re^n$, we have
		\begin{eqnarray*}
			&& \langle \A, \x \circ \y \circ \x \circ \y \rangle\\ &\equiv& \sum_{i_1, i_2 =1}^m \sum_{j_1, j_2 = 1}^n  \mathbb E\left((X_{i_1j_1}-\mu_{i_1j_1})(X_{i_2j_2}-\mu_{i_2j_2})\right)x_{i_1}y_{j_1}x_{i_2}y_{j_2}\\
			&=&\mathbb E\left( \sum_{i_1, i_2 =1}^m \sum_{j_1, j_2 = 1}^n  \mathbb (X_{i_1j_1}-\mu_{i_1j_1})(X_{i_2j_2}-\mu_{i_2j_2})x_{i_1}y_{j_1}x_{i_2}y_{j_2}\right)\\
			&=&\mathbb E\left( \vx^\top (X-{\bf \mu})\vy\right)^2\\
			&\ge & 0
		\end{eqnarray*}
		This completes the proof.
	\end{proof}
	
	Let $X^{(t)}$ represent the $t$th observed matrix data. We assume that each $X^{(t)}$ is an independent and identically distributed (iid) sample of the  random matrix $X=(X_{ij})\in\Re^{{m\times n}}$.
	Then  the estimated mean is $\bar{X}=(\bar{X}_{ij})\in\Re^{{m\times n}}$, where $\bar{X}_{ij}=\frac1T \sum_{t=1}^T X_{ij}^{(t)}$. A natural estimate for the covariance can be formulated as $\hat{\A}=(\hat{\sigma})_{i_1j_1i_2j_2} \in\Re^{{m\times n}\times {m\times n}}$, where
	\begin{equation}\label{eq:4cov_est}
		\hat{\sigma}_{ijkl}=\frac1T\sum_{t=1}^T \left(X_{ij}^{(t)} - \bar X_{ij}\right)\left(X_{kl}^{(t)} - \bar X_{kl}\right).
	\end{equation}

	\begin{Prop}\label{prop:psd_cov}
		The fourth-order covariance tensor defined in  \eqref{eq:4cov_est} is a positive semi-definite biquadratic tensor.
	\end{Prop}
	{\begin{proof}
			This result can be proven similar with that of Proposition~\ref{covtensor:PSD}, and for brevity, we omit the detailed proof here.
	\end{proof}}
	
	
	
	%
	
	{\subsection{The Elasticity Tensor in Solid Mechanics}}
	
	The field equations for a homogeneous, compressible, nonlinearly elastic material in two or three dimensions, pertaining to static problems in the absence of body forces, can be formulated as follows \cite{KS76}:
	\begin{equation}\label{eq:elasticity_tensor}
		a_{ijkl} \left(1+\nabla \vu \right) u_{k,lj} = 0,	
	\end{equation}
	Here, $u_i({\bf X})$ (with $i=1,2$ or $i=1,2,3$) represents the displacement field, where $\bf X$ denotes the coordinate of a material point in its reference configuration. Given  $n=2$ or $n=3$, the tensor   $\A = (a_{i_1j_1i_2j_2})\in BQ(n,n)$   signifies the elasticity tensor and
	\begin{equation*}
		a_{ijkl}({\bf F}) = a_{klij}({\bf F}) = \frac{\partial^2 W({\bf F})}{\partial F_{ij}\partial F_{kl}}, \  {\bf F} = 1+\nabla \vu.
	\end{equation*}
	{The tensor of elastic moduli is invariant with respect of
		permutations of indices as follows \cite{ZR16}
		\[a_{ijkl} = a_{jikl} = a_{ijlk}.\]
		For hyperelastic materials $a_{ijkl}$ has also the property
		\[a_{ijkl} = a_{klij}.\]}
	Then $\A$ is a weakly symmetric biquadratic tensor, {while its  symmetry may not be guaranteed.}   The elasticity tensor is the most well-known tensor in solid mechanics and engineering \cite{Ny85}.
	The above equations 	are strongly elliptic if and only if  for any $\x \in \Re^n$ and $\y \in \Re^n$, \eqref{equ:PSD} holds.
	Several methods were proposed for verifying strongly ellipticity of  the elasticity tensor \cite{CCZ21, HLW21,LCLL22,LLL19}.
	
	In solid mechanics, the Eshelby tensor is also a biquadratic tensor. The Eshelby inclusion problem is one of the hottest topics in modern solid mechanics \cite{ZHHZ10}.
	
	In 2009, Qi, Dai, and Han \cite{QDH09} proposed an optimization method for tackling  the problem of the
	strong ellipticity and also to give an algorithm for computing the most possible directions along which the strong ellipticity can fail.
	{Subsequently, a practical method for computing the largest M-eigenvalue of a fourth-order partially symmetric tensor was proposed by Wang et al. \cite{WQZ09}. Later,   \cite{LLL19, LCLL22} provided bounds of M-eigenvalues and strong ellipticity conditions for elasticity tensors.}

	\subsection{The   Riemann Curvature Tensor in Relativity Theory}
	
	In the {domain} of differential geometry, the Riemann curvature tensor, or  {Riemann-Christoffel} tensor (named after Bernhard Riemann and Elwin Bruno Christoffel), stands as the preeminent method for describing the curvature of Riemannian manifolds \cite{Ri06}. This tensor assigns a specific tensor to each point on a Riemannian manifold, thereby constituting a tensor field. Essentially, it serves as a local invariant of Riemannian metrics, quantifying the discrepancy in the commutation of second covariant derivatives. Notably, a Riemannian manifold possesses zero curvature if and only if it is flat, meaning it is locally isometric to Euclidean space.

	Let $(M,g)$ be a Riemannian manifold, and $\mathfrak{X}(M)$  be the space of all vector fields on $M$. The Riemann curvature tensor is defined as a map $\mathfrak{X}(M)\times \mathfrak{X}(M)\times \mathfrak{X}(M)\rightarrow \mathfrak{X}(M)$ by the following formula
	\begin{equation*}
		R(X,Y)Z=\nabla_X\nabla_YZ-\nabla_Y\nabla_XZ-\nabla_{[X,Y]}Z,
	\end{equation*}	
	where $\nabla$ is  {the Levi-Civita} connection, $[X,Y]$ is the Lie bracket of vector fields. It turns out that the right-hand side actually only depends on the value of the vector fields $X,Y,Z$ at a given point. Hence, $R$ is a (1,3)-tensor field. 
	By using the  tensor index notation and the Christoffel symbols, we have
	\begin{equation*}
		R^{i}_{jkl}= \partial_{k}\Gamma^{i}_{lj}-\partial_{l}\Gamma^{i}_{kj} +   \Gamma^{i}_{kp}\Gamma^{p}_{lj} - \Gamma^{i}_{lp}\Gamma^{p}_{kj},
	\end{equation*}
	where $\Gamma^{i}_{lj}$ are Christoffel symbols  of the first kind. For convenience, we could also write $R^{i}_{jkl}$ as $R_{ijkl}$.
	
	Denote $\R=(R_{i_1j_1i_2j_2})$ as the Riemann curvature tensor. Then it is a biquadratic tensor.
	The curvature tensor has the following symmetry properties \cite{XQW18}
	\[R_{ijkl}=-R_{jikl}=-R_{ijlk}=R_{klij},  R_{ijkl}+R_{iljk}+R_{iklj }=0.\]
	Therefore, the  Riemann curvature tensor is a weakly symmetric biquadratic tensor, {and it may not be symmetric}.

	
	\subsection{Bipartite Graphs and Graph Matching}
	Bipartite matching, or bipartite graph matching, is a fundamental problem in graph theory and combinatorics. It involves finding an optimal way to pair nodes from two disjoint sets in a bipartite graph, ensuring that no two pairs share nodes from the same set. This problem arises in various real-world scenarios, such as job assignment, task scheduling, and network flow optimization  \cite{YYLDZY16}.

	Consider a bipartite graph with two subgraphs $G_1=(V_1,E_1)$ and $G_2=(V_2,E_2)$,
	where $V_1$ and $V_2$ are {disjoint} sets of   points with $V_1=m$, $V_2=n$,  and $E_1$, $E_2$ are   sets of edges. The bipartite graph matching aims to find the best correspondence (also referred to as  matching) between $V_1$ and $V_2$  with the maximum matching score.
	Specifically, let $X=(x_{ij})\in\Re^{m\times n}$ be the assignment matrix between $V_1$ and $V_2$, i.e., $x_{ij}=1$ if $i\in V_1$ is assigned to $j\in V_2$ and $x_{ij}=0$ otherwise.
	Two edges $(i_1,j_1)\in E_1$ and $(i_2,j_2)\in E_2$ are said to be matched if $x_{i_1j_1}=x_{i_2j_2}=1$. Let $a_{i_1j_1i_2j_2}$ be the matching score between $(i_1,j_1)$ and $(i_2,j_2)$. Then $\A=(a_{i_1j_1i_2j_2})\in BQ(m,n)$ is a biquadratic tensor.  Assume that $m\ge n$. The graph matching problem takes the following form
	\begin{equation*}
		\max_{X\in \{0,1\}^{m\times n}} \ \sum_{(i_1,j_1)\in E_1}\sum_{(i_2,j_2)\in E_2} a_{i_1j_1i_2j_2} x_{i_1j_1}x_{i_2j_2} \text{ s.t. } X\1=\1.
	\end{equation*}
	It is commonly assumed that {$a_{i_1j_1i_2j_2}\ge 0$ and $a_{i_1j_1i_2j_2}=a_{i_2j_2i_1j_1}$, i.e.,} $\A$ is both weakly symmetric and nonnegative.
	
	Adjacency tensors and Laplacian tensors are basic tools in spectral hypergraph theory \cite{CD12, QL17}.  
	{Given a bipartite graph,} the corresponding adjacency tensors and Laplacian tensors  {can also be formulated as} biquadratic tensors.


	\subsection{Biquadratic Polynomials and Polynomial Theory}
	
	Given a  biquadratic {tensor} $\A\in BQ(m,n)$, the	biquadratic polynomial   ${f(\x, \y) \equiv}$ $\langle \A, \x \circ \y \circ \x \circ \y \rangle$ constitutes a significant branch within the realm of polynomial theory.

	In 2009, Ling et al. \cite{LNQY10} studied the  so-called biquadratic optimization over unit spheres,
	\begin{equation*}
		\min_{\vx\in\Re^m,\vy\in\Re^n}\quad  \langle \A, \x \circ \y \circ \x \circ \y \rangle \text{  and  } \|\vx\|=1, \|\vy\|=1.
	\end{equation*}
	Then they    presented various approximation methods based on  semi-definite  programming  relaxations.
	{In 2012, Yang and Yang \cite{YY12} studied the biquadratic optimization with wider constrains. They relaxed the original problem to its semi-definite  programming problem, discussed the approximation ratio between them, and showed that the relaxed problem is tight under some conditions.}
	In 2016, {O'Rourke} et al. \cite{RSRS18} subsequently leveraged the theory of biquadratic optimization to address the joint signal-beamformer optimization problem and proposed a  {semi-definite} relaxation to formulate a more manageable version of the problem.  This demonstrates the significance of investigating biquadratic optimization.

	One important problem in polynomial theory is the Hilbert's 17th Problem. In 1900, German mathematician David Hilbert  listed   23 unsolved mathematical challenges proposed   at the International Congress of Mathematicians. Among these problems, Hilbert's 17th Problem stands out as it relates to the representation of polynomials. Specifically, Hilbert's 17th Problem asks whether every nonnegative polynomial can be represented as a sum of squares (SOS) of rational functions.
	The Motzkin polynomial is the first example of a nonnegative  polynomial that is not an SOS.
	Hilbert proved that in several cases, every  {nonnegative} polynomial is an SOS, including
	univariate polynomials, quadratic polynomials in any number of variables, bivariate quartic polynomials  (i.e., polynomials of degree 4 in two variables)  {\cite{Hi88}}. 
	In general, a quartic nonnegative polynomial may not be represented as an SOS, such as the  {Robinson polynomial and the} Choi-Lam polynomial.
	The question at hand is whether a  {given} biquadratic  nonnegative polynomial can be represented as an SOS?
	
	If $\A\in BQ(m,n)$ is  a nonnegative diagonal biquadratic tensor, then we have ${f(\x, \y) \equiv} \sum_{i=1}^m\sum_{j=1}^n a_{ijij} x_i^2x_j^2$, which is an SOS expression.
	In the following, we present a sufficient condition for the nonnegative biquadratic polynomial   to be SOS.
	
	\begin{Prop}\label{prop_SOS}
		Given a  biquadratic tensor  $\A=(a_{i_1j_1i_2j_2})\in BQ(m,n)$ with $m,n\ge 2$ and the	biquadratic polynomial   ${f(\x, \y) \equiv} \langle \A, \x \circ \y \circ \x \circ \y \rangle$. Suppose that there exist matrices $A^{(k)}=\left(a_{ij}^{(k)}\right)\in\mathbb R^{m\times n}$ such that $\A=\sum_{k=1}^K A^{(k)} \otimes A^{(k)}$, i.e., $a_{i_1j_1i_2j_2} = \sum_{k=1}^K a_{i_1j_1}^{(k)}a_{i_2j_2}^{(k)}$, then $f$ can be  {expressed} as an SOS.
	\end{Prop}
	\begin{proof}
		Suppose that   $\A=\sum_{k=1}^K A^{(k)} \otimes A^{(k)}$. Then for any $\vx\in\Re^m$ and $\vy\in\Re^n$, we have
		\begin{eqnarray*}
			f(\vx,\vy) &=& \langle \A, \x \circ \y \circ \x \circ \y \rangle\\
			&=& \sum_{k=1}^K \left(\vx^\top A^{(k)} \vy\right)^2,
		\end{eqnarray*}
		which is an SOS. This completes the proof.
	\end{proof}
	
	In fact,  the nonnegative diagonal biquadratic tensor $\A=(a_{i_1j_1i_2j_2})$ may be rewritten as
	$$ \A = \sum_{i=1}^m \sum_{j=1}^n \sqrt{a_{ijij}}E_{ij} \otimes E_{ij},$$
	where $E_{ij}\in \{0,1\}^{m\times n}$ contains exactly  one nonzero entry, specifically located at the $(i,j)$th position.
	
	Consider a radar datacube collected over  $L$ pulses with $N$ fast time samples per pulse in $M$ spatial bins. The biquadratic tensor in \cite{RSRS18} could be written as
	\begin{equation*}
		\C = \sum_{q=1}^Q  {\Gamma}_q \otimes {\Gamma}_q^H,
	\end{equation*}
	where $\Gamma_q\in\mathbb C^{NML\times N}$ is the  spatiodoppler response matrix, $Q$ is the number of independent clutter patches.
	If $\Gamma_q$ is in the real field, then $\C$ satisfies the assumptions in Proposition \ref{prop_SOS}.

	\section{Eigenvalues of Biquadratic Tensors}

	Suppose that $\A = (a_{i_1j_1i_2j_2}) \in BQ(m, n)$.   A   {real} number $\lambda$ is called an  {M-eigenvalue} of $\A$ if there are  {real} vectors  $\x = (x_1, {\dots,} x_m)^\top \in {\Re}^m, \y = (y_1, {\dots,} y_n)^\top \in {\Re}^n$ such that the following equations are satisfied:
	For  $i {\in [m]}$,
	\begin{equation} \label{e5}
		\sum_{i_1=1}^m \sum_{j_1, j_2=1}^n a_{i_1j_1ij_2}x_{i_1}y_{j_1}y_{j_2} +	\sum_{i_2=1}^m \sum_{j_1, j_2=1}^n a_{ij_1i_2j_2}y_{j_1}x_{i_2}y_{j_2} = 2\lambda x_i;
	\end{equation}
	{for} $j {\in [n]}$,
	\begin{equation} \label{e6}
		\sum_{i_1,i_2=1}^m\sum_{j_1=1}^n a_{i_1j_1i_2j}x_{i_1}y_{j_1}x_{i_2} +	\sum_{i_1,i_2=1}^m\sum_{j_2=1}^n a_{i_1ji_2j_2}x_{i_1}x_{i_2}y_{j_2} = 2\lambda y_j;
	\end{equation}
	and
	\begin{equation} \label{e7}
		\x^\top \x ={ \y^\top \y} = 1.
	\end{equation}
	Then $\x$ and $\y$ are called the corresponding  {M-eigenvectors}.   
	M-eigenvalues were introduced {for symmetric biquadratic tensors}  in 2009 by Qi, Dai and Han \cite{QDH09},  i..e,
	for  $i {\in [m]}$,
	\begin{equation} \label{e5_sym}
		\sum_{i_2=1}^m \sum_{j_1, j_2=1}^n a_{ij_1i_2j_2}y_{j_1}x_{i_2}y_{j_2} =  \lambda x_i;
	\end{equation}
	{for} $j {\in [n]}$,
	\begin{equation} \label{e6_sym}
		\sum_{i_1,i_2=1}^m\sum_{j_2=1}^n a_{i_1ji_2j_2}x_{i_1}x_{i_2}y_{j_2} =  \lambda y_j.
	\end{equation}
	Subsequently, several numerical methods, including the WQZ method \cite{WQZ09}, semi-definite relaxation method \cite{LNQY10,YY12}, and the shifted inverse  power method \cite{ZLS24,WCWY23}, were proposed.   It is easy to see that if $\A$ is symmetric, then (\ref{e5}) and (\ref{e6}) reduce to the definition of M-eigenvalues {in equations \eqref{e5_sym} and \eqref{e6_sym}} in \cite{QDH09}. Hence, we
	extend M-eigenvalues and M-eigenvectors to nonsymmetric biquadratic tensors here.
	
	We have the following theorem.
	
	\begin{Thm} \label{T3.1}
		Suppose that $\A \in BQ(m, n)$.  Then $\A$ always have M-eigenvalues.  Furthermore, $\A$ is positive semi-definite if and only if all of its M-eigenvalues are nonnegative, $\A$ is positive definite if and only if all of its M-eigenvalues are positive.
	\end{Thm}
	\begin{proof}
		Consider the optimization problem
		\begin{equation} \label{opt}
			\min \{ f(\x, \y, \x, \y) : \x^\top \x = 1, \y^\top \y = 1 \}.
		\end{equation}
		The feasible set is compact.  The objective function is continuous.   Thus, an optimizer of this problem exists.  Furthermore, the problem satisfies the linear independence constraint qualification.  Hence, at the optimizer, the optimal conditions hold.   This implies that (\ref{e5}) and (\ref{e6}) hold, i.e., an M-eigenvalue always exists.   By (\ref{e5}) and (\ref{e6}), we have
		$$\lambda = f(\x, \y, \x, \y).$$
		Hence, $\A$ is positive semi-definite if and only if all of its M-eigenvalues are nonnegative, $\A$ is positive definite if and only if all of its M-eigenvalues are positive.
	\end{proof}

	Let $\I = (I_{i_1j_1i_2j_2})\in BQ(m,n)$, where 	$$I_{i_1j_1i_2j_2}=\left\{\begin{array}{cl}
		1,& \text{ if } i_1=i_2 \text{ and }j_1=j_2,\\
		0, & \text{otherwise.}
	\end{array}\right.
	$$
	$\I$ is referred to as $M$-identity tensor in {\cite{DLQY20,WWL20}.}
	In the following proposition, we see that $\I$ acts as the identity tensor for biquadratic tensors.
	
	\begin{Prop}\label{prop:A-lmdI}
		Suppose that $\A \in BQ(m, n)$.  Then $\mu -\lambda$ {(resp. $\lambda-\mu$)}  is an M-eigenvalue of $\A - \lambda \I$ {(resp. $\lambda \I-\A)$} if and only if $\mu$ is an M-eigenvalue of $\A$.
	\end{Prop}
	\begin{proof}
		{By} the {definitions} of M-eigenvalues and $\I$, we have the conclusion.
	\end{proof}
	
	Suppose that $\A = \left(a_{i_1j_1i_2j_2}\right) \in BQ(m, n)$.  Then we call $a_{ijij}$ diagonal entries of $\A$ for $i {\in [m]}$ and $j {\in [n]}$.  The other entries of $\A$ are called off-diagonal entries of $\A$.  If all the off-diagonal entries are zeros, then $\A$ is called a diagonal biquadratic tensor.   If $\A$ is diagonal, then $\A$ has $mn$  M-eigenvalues, which are its diagonal elements,
	with corresponding  vectors  
	{$\vx = \ve_i(m)$ and $\vy =\ve_j(n)$,}
	where $\ve_i(m)$ is the $i$th unit vector in ${\Re}^m$, and $\ve_j(n)$ is the $j$th unit vector in ${\Re}^n$,  as their M-eigenvectors.   However, different from cubic tensors, in this case, $\A$ may have some other M-eigenvalues and M-eigenvectors.  In fact, for a diagonal biquadratic tensor $\A$, (\ref{e5}) and (\ref{e6}) have the form:
	for $i {\in [m]}$,
	\begin{equation} \label{e5.1}
		\sum_{j=1}^n a_{ijij}y_j^2x_i = \lambda x_i;
	\end{equation}
	and for $j {\in [n]}$,
	\begin{equation} \label{e6.1}
		\sum_{i=1}^m a_{ijij}x_i^2y_j = \lambda y_j.
	\end{equation}

	We now present an example with $m=n=2$.
	
	\begin{example}-- A diagonal biquadratic tensor $\A$  may possess M-eigenvalues that are distinct from its diagonal elements.
		Suppose that $a_{1111}=1$, $a_{1212}=\alpha$, $a_{2121}=\beta$, and $a_{2222} = \gamma$. Then  by equations \eqref{e5.1} and \eqref{e6.1}, we have
		\begin{eqnarray*}
			\left\{
			\begin{array}{c}
				x_1\left( y_1^2+\alpha y_2^2-\lambda\right)=0,\\
				x_2\left( \beta y_1^2+\gamma y_2^2-\lambda\right)=0,\\
				y_1\left( x_1^2+\beta x_2^2-\lambda\right)=0,\\
				y_2\left(\alpha x_1^2+\gamma x_2^2-\lambda\right)=0.\\
			\end{array}
			\right.
		\end{eqnarray*}
		Let $x_1^2+x_2^2=1$ and $y_1^2+y_2^2=1$.  Suppose that $x_1,x_2,y_1,y_2$ are  nonzero numbers. Otherwise, if {any of these numbers are} zero, the eigenvalue shall be equal to one of the diagonal elements.  Then  we have an M-eigenvalue $\lambda = \frac{\gamma-\alpha \beta }{\gamma+1-\alpha-\beta}$ and the  corresponding M-eigenvectors satisfy
		\begin{equation*}
			x_1^2= \frac{\gamma- \beta }{\gamma+1-\alpha-\beta},  x_2^2 = \frac{1-\alpha}{\gamma+1-\alpha-\beta},
		\end{equation*}
		and
		\begin{equation*}
			y_1^2= \frac{\gamma- \alpha }{\gamma+1-\alpha-\beta},  y_2^2 = \frac{1-\beta}{\gamma+1-\alpha-\beta}.
		\end{equation*}
		For instance, if $\alpha=\beta=-1$ and $\gamma = 1$, we have $\lambda = 0$, which is not a diagonal element of $\A$, and the corresponding M-eigenvectors satisfy  $x_1^2=x_2^2=y_1^2=y_2^2=\frac12$.
	\end{example}

	Let $\A = \left(a_{i_1j_1i_2j_2}\right) \in BQ(m, n)$.   Then its diagonal entries form an $m \times n$ rectangular matrix $D = \left(a_{ijij}\right){\in\Re^{m\times n}}$.
	
	By (\ref{e5.1}) and (\ref{e6.1}), we have the following proposition.
	
	\begin{Prop}
		Suppose that $\A = \left(a_{i_1j_1i_2j_2}\right) \in BQ(m, n)$ is a diagonal biquadratic tensor, then $\A$ has $mn$  M-eigenvalues, which are its diagonal elements {$a_{ijij}$},
		with corresponding  vectors
		{{$\vx=\ve_i(m)$ and $\vy=\ve_j(n)$}}
		as their M-eigenvectors.   Furthermore, though $\A$ may have some other M-eigenvalues, they are still in the convex hull of {some} diagonal entries.
	\end{Prop}
	\begin{proof}
		When  {$\vx=\ve_i(m)$ and $\vy=\ve_j(n)$}, we may verify that (\ref{e5.1}) and (\ref{e6.1}) hold with $\lambda=a_{ijij}$. This shows the first conclusion.

		Furthermore, let $I_x=\{i: x_i\neq 0\}$, $I_y=\{j: y_j\neq 0\}$, and $D_{I}  \in \Re^{|I_x|\times |I_y|}$.
		{Here, $D_{I}$ is composed of the rows $I_x$ and columns $I_y$ of $D =\left(a_{ijij}\right){\in\Re^{m\times n}}$.}
		Then equations  (\ref{e5.1}) and (\ref{e6.1}) can be reformulated as
		\begin{equation*}
			D_I \vy_{I_y}^{[2]}=\lambda {\bf 1}_{|I_x|},  D_I^\top \vx_{I_x}^{[2]}=\lambda {\bf 1}_{|I_y|}.
		\end{equation*}
		Here,  $\vx_{I_x}^{[2]} = \{(x_i^2)\}_{i\in I_x}$ and $\vy_{I_y}^{[2]} = \{(y_j^2)\}_{j\in I_y}$.
		Therefore, given the index sets $I_x$ and $I_y$, as long as 
		{${\bf 1_{|I_x|}}$ and ${\bf 1_{|I_y|}}$  are}
		in the convex hull of $D_I$ and $D_I^\top$, {respectively,} then there is an M-eigenvalue in the convex hull of the 
		{entries} in $D_I$.
		
		This completes the proof.
	\end{proof}

	We have the following theorem for biquadratic tensors.  	
	
	\begin{Thm}
		A necessary condition for a biquadratic tensor $\A \in BQ(m, n)$ to be positive semi-definite is that all of its diagonal entries are nonnegative.   A necessary condition for $\A$ to be positive definite is that all of its diagonal entries are positive.   If $\A$ is a diagonal biquadratic tensor, then these conditions are sufficient.
	\end{Thm}
	\begin{proof}
		Let $\A = (a_{i_1j_1i_2j_2})$ and
		$${f(\x, \y) =} \langle \A, \x \circ \y \circ \x \circ \y \rangle \equiv \sum_{i_1, i_2 =1}^m \sum_{j_1, j_2 = 1}^n a_{i_1j_1i_2j_2}x_{i_1}y_{j_1}x_{i_2}y_{j_2}.$$
		Then $f(\ve_i(m), \ve_j(n)) = {a_{i jij}}$
		for $i {\in [m]}$ and $j {\in [n]}$.
		This leads to the first two conclusions.   If $\A$ is a diagonal biquadratic tensor, then
		$$f(\x, \y) = \sum_{i=1}^m \sum_{j=1}^n a_{ijij}x_i^2y_j^2.$$
		The last conclusion follows.
	\end{proof}

	
	

	\section{Gershgorin-Type Theorem and Diagonally Dominated Biquadratic Tensors}

	Recall that for square matrices, there is a Gershgorin theorem, 
	{from which we may show that (strictly) diagonally dominated  matrices are  {positive semi-definite (definite)}.}
	These have been successfully generalized to cubic tensors \cite{Qi05, QL17}.    Then, can these be generalized to {nonsymmetric} biquadratic tensors?   To answer this question, we have to understand the ``rows'' and ``columns'' of a biquadratic tensor.
	
	Suppose that $\A = \left(a_{i_1j_1i_2j_2}\right) \in BQ(m, n)$.  Then $\A$ has $m$ rows.   At the $i$th row of $\A$, there are $n$ diagonal entries $a_{ijij}$ for $j  {\in [n]}$, and totally $2mn^2$ entries  {$a_{i_1j_1ij_2}$ and} $a_{ij_1i_2j_2}$ for $i_1, i_2  {\in [m]}$ and $j_1, j_2 {\in [n]}$.     {We use}  the notation $\bar a_{ij_1i_2j_2}$.   Let
	$\bar a_{i_1j_1i_2j_2} = a_{i_1j_1i_2j_2}$ if $a_{i_1j_1i_2j_2}$ is an off-diagonal entry, and $\bar a_{i_1j_1i_2j_2} = 0$ if $a_{i_1j_1i_2j_2}$ is a diagonal entry.		Then,  {in the $i$th row, the}   diagonal entry {$a_{ijij}$}   has the responsibility to dominate the
	{$4mn$} entries {$\bar a_{i_1jij_2}$, $\bar a_{i_1j_1ij}$,} $\bar a_{iji_2j_2}$  and $\bar a_{ij_1i_2j}$ for $i_1, i_2 \in [m]$ and $j_1, j_2 \in [n]$.  
	Therefore, for $i {\in [m]}$ and $j {\in [n]}$, we define
	\begin{equation}\label{equ:rij}
		r_{ij}= \frac14\sum_{i_1=1}^m \left(\sum_{j_2=1}^n \left|\bar a_{i_1jij_2}\right| + \sum_{j_1=1}^n \left|\bar a_{i_1j_1ij}\right|\right) + \frac14\sum_{i_2=1}^m \left(\sum_{j_2=1}^n \left|\bar a_{iji_2j_2}\right| + \sum_{j_1=1}^n \left|\bar a_{ij_1i_2j}\right|\right).
	\end{equation}
	
	If $\A$ is weakly symmetric, then \eqref{equ:rij} reduces to
	\[r_{ij} = \frac12\sum_{i_2=1}^m \left(\sum_{j_2=1}^n \left|\bar a_{iji_2j_2}\right| + \sum_{j_1=1}^n \left|\bar a_{ij_1i_2j}\right|\right).\]
	If $\A$ is  symmetric, then \eqref{equ:rij} reduces to
	\[r_{ij} = \sum_{i_2=1}^m  \sum_{j_2=1}^n \left|\bar a_{iji_2j_2}\right|.\]



	For M-eigenvalues, here we have the following theorem,  which generalizes the Gershgorin theorem  of matrices  to biquadratic tensors.
	
	\begin{Thm} \label{Ger}
		Suppose that $\A = \left(a_{i_1j_1i_2j_2}\right) \in BQ(m, n)$. Let {$r_{ij}$   for $i {\in [m]}$ and $j {\in [n]}$, be defined  by equation  \eqref{equ:rij}.}
		Then any M-eigenvalue $\lambda$ of $\A$ lies in one of {the} following $m$ intervals:
		\begin{equation} \label{Ger1}
			\left[\min_{j=1,\dots,n} \left\{a_{ijij} -{ r_{ij}}\right\},  \max_{j=1,\dots,n}\left\{a_{ijij} + {r_{ij}}\right\} \right]
		\end{equation}
		for $i {\in [m]}$, and one of the following $n$ intervals:
		\begin{equation} \label{Ger2}
			\left[\min_{i=1,\dots,m} \left\{a_{ijij} - {r_{ij}}\right\}, \max_{i=1,\dots,m} \left\{a_{ijij} + {r_{ij}}\right\}\right],
		\end{equation}
		for $j {\in [n]}$.
	\end{Thm}
	\begin{proof}
		Suppose that $\lambda$ is an M-eigenvalue of $\A$, with M-eigenvectors $\vx$ and $\vy$.   Assume that {$x_i\neq 0$} is the component of $\vx$ with the largest absolute value.  From (\ref{e5}), we have
		\begin{eqnarray*}
			\lambda & =  & \sum_{j=1}^n {a_{ijij}y_j^2} + {{1\over 2} \sum_{i_1=1}^m \sum_{j_1, j_2=1}^n  \bar a_{i_1j_1ij_2}{x_{i_1} \over x_i}y_{j_1}y_{j_2} +  {1 \over 2}}\sum_{i_2=1}^m \sum_{j_1, j_2=1}^n  \bar a_{ij_1i_2j_2}y_{j_1} {x_{i_2} \over x_i}y_{j_2} \\
			& \le &  \sum_{j=1}^n{a_{ijij}y_j^2}+ {{1\over 2} \sum_{i_1=1}^m \sum_{j_1, j_2=1}^n \left| \bar a_{i_1j_1ij_2}\right| \left|y_{j_1}y_{j_2}\right| + {1 \over 2}}\sum_{i_2=1}^m \sum_{j_1, j_2=1}^n  \left|\bar a_{ij_1i_2j_2}\right| \left|y_{j_1}y_{j_2}\right| \\
			& \le &  \sum_{j=1}^n {a_{ijij}y_j^2} + {{1\over 4} \sum_{i_1=1}^m \sum_{j_1, j_2=1}^n \left| \bar a_{i_1j_1ij_2}\right| \left(y_{j_1}^2+y_{j_2}^2\right) + {1 \over 4}}\sum_{i_2=1}^m \sum_{j_1, j_2=1}^n  \left|\bar a_{ij_1i_2j_2}\right| \left(y_{j_1}^2 + y_{j_2}^2 \right) \\
			&=&   \sum_{j=1}^n y_j^2 \left[a_{ijij} + \frac14\sum_{i_1=1}^m \left(\sum_{j_2=1}^n \left|\bar a_{i_1jij_2}\right| + \sum_{j_1=1}^n \left|\bar a_{i_1j_1ij}\right|\right)\right.\\
			&& + \frac14\left.\sum_{i_2=1}^m \left(\sum_{j_2=1}^n \left|\bar a_{iji_2j_2}\right| + \sum_{j_1=1}^n \left|\bar a_{ij_1i_2j}\right|\right)\right] \\
			& \le & \max_{j=1,\cdots, n} \left[a_{ijij} + r_{ij}\right],
		\end{eqnarray*}	
		for $i{\in [m]}$.  The other inequalities of (\ref{Ger1}) and (\ref{Ger2}) can be derived similarly.
		
		{This completes the proof.}
	\end{proof}
	
	In the symmetric case, the inclusion intervals and bounds of M-eigenvalues have been presented in \cite{LLL19,LCLL22,WWL20, Zh23}.   Here, Theorem \ref{Ger} covers the nonsymmetric case too.

	If the entries of $\A$ satisfy
	\begin{equation}\label{dom1}
		a_{ijij}\ge r_{ij}\equiv  {\frac14\sum_{i_1=1}^m \left(\sum_{j_2=1}^n \left|\bar a_{i_1jij_2}\right| + \sum_{j_1=1}^n \left|\bar a_{i_1j_1ij}\right|\right) + \frac14}\sum_{i_2=1}^m \left(\sum_{j_2=1}^n \left|\bar a_{iji_2j_2}\right| + \sum_{j_1=1}^n \left|\bar a_{ij_1i_2j}\right|\right),
	\end{equation}
	for all $i {\in [m]}$ and $j{\in [n]}$,
	then $\A$ is called a diagonally dominated biquadratic tensor.  If {the} strict inequality holds for all these {$mn$} inequalities, then $\A$ is called a strictly diagonally dominated biquadratic tensor.

	Suppose that $\A = \left(a_{i_1j_1i_2j_2}\right) \in SBQ(m, n)$. Then equation \eqref{dom1}   can be {simplified} as
	\begin{equation}\label{dom_sym}
		a_{ijij} \ge \sum_{i_2=1}^m \sum_{j_2=1}^n \left|\bar a_{iji_2j_2}\right|,
	\end{equation}
	for all $i{\in [m]}$ and $j{\in [n]}$.
	
	\begin{Cor}\label{cor_dominated_PD}
		A diagonally dominated biquadratic  tensor is positive semi-definite.   A strictly diagonally  dominated  biquadratic  tensor is positive definite.
	\end{Cor}
	\begin{proof}
		This result follows directly from Theorems~\ref{Ger} and \ref{T3.1}.
	\end{proof}

	{

		\section{B-Biquadratic Tensors and M-Biquadratic Tensors}
		
		Several structured tensors, including B$_0$, B-tensors, Z-tensors, M-tensors and strong M-tensors have been studied.   See \cite{QL17} for this.
		It has been established that even order
		symmetric B$_0${-tensors} and even order symmetric M-tensors are positive semi-definite,  even order symmetric B{-tensors} and even order symmetric strong M-tensors are positive definite {\cite{QS14, QL17}.}
		We now extend such  structured cubic tensors and the above properties
		to biquadratic tensors.
		
		Let $\A = \left(a_{i_1j_1i_2j_2}\right) \in BQ(m, n)$.  Suppose that
		for $i {\in [m]}$ and $j {\in [n]}$, we have
		\begin{equation}\label{B1}
			\sum_{i_1=1}^m \left(\sum_{j_2=1}^n a_{i_1jij_2} + \sum_{j_1=1}^n  a_{i_1j_1ij}\right) +\sum_{i_2=1}^m \left(\sum_{j_2=1}^n  a_{iji_2j_2} + \sum_{j_1=1}^n  a_{ij_1i_2j}\right) \ge 0,
		\end{equation}
		and for $i, i_2 {\in [m]}$ and $j, j_1, j_2 {\in [n]}$, we have
		\begin{eqnarray}
			\nonumber	&& {1 \over 4mn}\left[ {\sum_{i_1=1}^m \left(\sum_{j_2=1}^n a_{i_1jij_2} + \sum_{j_1=1}^n  a_{i_1j_1ij}\right) + }\sum_{i_2=1}^m \left(\sum_{j_2=1}^n  a_{iji_2j_2} + \sum_{j_1=1}^n  a_{ij_1i_2j}\right)  \right]\\
			&\ge &\max \left\{ \bar a_{iji_2j_2}, \bar a_{i_1jij_2}, \bar a_{ij_1i_2j}, \bar a_{i_1j_1ij} \right\}.\label{B2}
		\end{eqnarray}
		Then we say that $\A$ is a B$_0${-}biquadratic tensor.  If {all} strict inequalities hold in (\ref{B1}) and (\ref{B2}), then we say that $\A$ is a B{-}biquadratic tensor.

		A biquadratic tensor $\A$ in $BQ(m,n)$ is called a Z-biquadratic tensor if all of its off-diagonal entries are nonpositive.
		If $\A$ is a Z-biquadratic tensor, then it can be written as $\A=\alpha \I-\B$, where $\I$ is the identity tensor in $BQ(m,n)$, and $\B$ is a nonnegative biquadratic tensor.   
		By the discussion in the last section, $\B$ has an M-eigenvalue.
		Denote ${\lambda_{\max}}(\B)$ as the largest M eigenvalue of $\B$. 
		If $\alpha \ge {\lambda_{\max}}(\B)$, then $\A$ is called an M-biquadratic tensor.   If $\alpha > {\lambda_{\max}}(\B)$, then $\A$ is called a strong M-biquadratic tensor.
		
		If $\B \in SBQ(3, 3)$ and is nonnegative, {it follows from} Theorem 6 of \cite{DLQY20} {that} $\lambda_{\max}(\B) = \rho_M(\B)$, where $\rho_M(\B)$ is the M-spectral radius of $\B$, i.e., the largest absolute value of the M-eigenvalues of $\B$.   Checking the proof of Theorem 6 of \cite{DLQY20}, we may find that this is true for $\B \in SBQ(m, n)$, $m, n \ge 2$ and $B$ being nonnegative.   If $\B$ is nonnegative but not symmetric, this conclusion is still an open problem at this moment.
		
		The following proposition is a direct generalization of Proposition 5.37 of \cite{QL17} for cubic tensors to biquadratic tensors.
		
		\begin{Prop}\label{prop_Z-B-tensor}
			Suppose $\A\in BQ(m,n)$ is a Z-biquadratic tensor. Then
			\begin{itemize}
				\item[(i)] $\A$ is diagonally dominated if and only if $\A$ is a B$_0$-biquadratic tensor;
				\item[(ii)] $\A$ is strictly diagonally dominated if and only if $\A$ is a B-biquadrtic tensor.
			\end{itemize}
		\end{Prop}
		\begin{proof}
			By  the fact that all off-diagonal entries of $\A$ are nonpositive, for all $i\in[m]$ and $j\in [n]$, we have
			\begin{eqnarray}
				&&{1 \over 4}\left[{\sum_{i_1=1}^m \left(\sum_{j_2=1}^n a_{i_1jij_2} + \sum_{j_1=1}^n  a_{i_1j_1ij}\right) + }\sum_{i_2=1}^m \left(\sum_{j_2=1}^n  a_{iji_2j_2} + \sum_{j_1=1}^n  a_{ij_1i_2j}\right)\right]\\
				\nonumber	&=&	a_{ijij}-   {1 \over 4}\left[\sum_{i_1=1}^m \left(\sum_{j_2=1}^n {|\bar a_{i_1jij_2}} + \sum_{j_1=1}^n  {|\bar a_{i_1j_1ij}|}\right) + \sum_{i_2=1}^m \left(\sum_{j_2=1}^n  {|\bar a_{iji_2j_2}|} + \sum_{j_1=1}^n  {|\bar a_{ij_1i_2j}|}\right)\right].
			\end{eqnarray}
			Therefore,   \eqref{dom1} and \eqref{B1}  are equivalent.
			
			If $\A$ is  diagonally dominated, then  \eqref{dom1} holds true. Hence, \eqref{B1} also holds.  Since the left hand side of \eqref{B2} is nonnegative, and the right hand side is nonpositive, we have \eqref{B2}   holds.
			
			On the other hand, suppose that $\A$ is a B$_0$-biquadratic tensor. Then  \eqref{B1}   holds, thus \eqref{dom1} holds true.
			
			Similarly, we could show the last statement of this  {proposition}. This completes the proof.
		\end{proof}

		By Proposition \ref{prop:A-lmdI}, we have the following proposition.
		
		\begin{Prop}\label{prop6.1}
			An M-biquadratic tensor is positive semi-definite.
			A strong M-biquadratic tensor is positive definite.
		\end{Prop}
		\begin{proof}
			Suppose that $\A=\alpha \I-\B$, where $\alpha \ge{\lambda_{\max}}(\B)$,  $\I$ is the identity tensor in $BQ(m,n)$, and $\B$ is a nonnegative biquadratic tensor.  It follows from Proposition~\ref{prop:A-lmdI} that every M-eigenvalue of $\A$ satisfies $$\alpha-\mu\ge {\lambda_{\max}}(\B)-\mu\ge 0,$$
			where $\mu$ is an M-eigenvalue of $\B$.
			{This combined with Theorem~\ref{T3.1} shows the first statement of this proposition.}
			
			Similarly, we could show the last statement of this {proposition}. This completes the proof.
		\end{proof}
		
		\begin{Prop}\label{prop_M-B-tensor}
			Let $\A=(a_{i_1j_1i_2j_2}) \in BQ(m,n)$ be a Z-biquadratic tensor.  $\A$ is positive semi-definite if and only if it is an M-biquadratic tensor.  Similarly, $\A$ is positive definite if and only if it is a strong M-biquadratic tensor.
		\end{Prop}
		\begin{proof}
			The necessity part follows from Proposition~\ref{prop6.1}. We will now demonstrate the sufficiency aspect.
			
			Let 
			$$\B=\gamma{\I} - \A, \text{ where }\gamma = \max\left\{\max_{ij}a_{ijij},{\lambda_M(\A)} \right\}.$$
			Then $\B$ is nonnegative. Furthermore, by Proposition~\ref{prop:A-lmdI},   we have $${\lambda_{\max}}(\B)=\max\{\gamma-\lambda: \lambda \text{ is an M-eigenvalue of } \A\}\le \gamma.$$
			Here, the last inequality follows from the fact that $\A$ is positive semidefinite and all the M-eigenvalues are nonnegative.
			This shows that $\A=\gamma I - \B$ is  an M-biquadratic tensor.
			
			Similarly, we could show the last statement of this {proposition}. This completes the proof.
		\end{proof}
		
		\begin{Cor}\label{Cor_M-B-tensor}
			Let $\A=(a_{i_1j_1i_2j_2}) \in BQ(m,n)$ be a Z-biquadratic tensor.  If all diagonal entries of $\A$ are nonnegative and there exist two positive diagonal matrices $D\in\Re^{m\times m}$ and $F\in\Re^{n\times n}$ such that $\C:=\A\times_1 D \times_2 F \times_3 D \times_4 F$ is diagonally dominated, then $\A$ is an M-biquadratic tensor.
			Similarly, if all diagonal entries of $\A$ are positive and there exist two positive diagonal matrices $D\in\Re^{m\times m}$ and $F\in\Re^{n\times n}$ such that $\C:=\A\times_1 D \times_2 F \times_3 D \times_4 F$ is strictly  diagonally dominated, then $\A$ is a strong M-biquadratic tensor.
		\end{Cor}
		\begin{proof}
			By Corollary~\ref{cor_dominated_PD}, $\C$ is positive semi-definite. By Theorem~7 in \cite{JYZ17}, an M-eigenvalue of $\C$  is also an M-eigenvalue of $\A$.  Therefore, $\A$ is also positive semi-definite. This result follows directly from Proposition~\ref{prop_M-B-tensor} .
			
			Similarly, we could show the last statement of this  corollary. This completes the proof.
		\end{proof}

		Let  
		{$[m]\otimes [n]=\{(i,j):  i \in [m], j\in[n]\}$.} For any $J=({J_x,J_y})\in [m]\otimes [n]$, denote $\I^J=(I_{i_1j_1i_2j_2})$ as a biquadratic  tensor in $BQ(m,n)$, where $I_{i_1j_1i_2j_2}=1$ if $i_1,i_2\in {J_x}$ and $j_1,j_2\in {J_y}$, and $I_{i_1j_1i_2j_2}=0$ otherwise.
		Then for any nonzero vectors $\vx\in\Re^m$ and  $\vx\in\Re^m$, we have
		\begin{equation*}
			\langle\I^J, \vx\circ \vy\circ\vx\circ\vy\rangle = \left(\vx_{J_x}^\top\vx_{J_x}^\top \right) \left(\vy_{J_y}^\top\vy_{J_y}^\top \right)\ge0,
		\end{equation*}	
		where $\vx_{J_x} = (\tilde x_i) \in\Re^m$ with $\tilde x_i=1$ if $i\in J_x$ and $\tilde x_i=0$ otherwise.

		Similar with Theorem 5.38 in \cite{QL17} for cubic tensors, we could establish the following interesting decomposition theorem.
		
		\begin{Thm}\label{Thm:B_decomp}
			Suppose $\B=(b_{i_1j_1i_2j_2})\in SBQ(m,n)$ is a symmetric B$_0$-quadratic tensor, i.e.,
			for $i ,i_2{\in [m]}$ and $j,j_2 {\in [n]}$, we have
			\begin{equation}\label{sym_Btensor}
				\sum_{i_2=1}^m \sum_{j_2=1}^n  b_{iji_2j_2}  \ge 0 \text{ and } \frac{1}{mn} \sum_{i_2=1}^m \sum_{j_2=1}^n  b_{iji_2j_2}   \ge  \bar b_{iji_2j_2}.
			\end{equation}
			Then either $\B$ is a diagonal dominated symmetric M-biquadratic tensor itself, or it can be decomposed as
			\begin{equation}\label{equ_decomp}
				\B={\mathcal M} +\sum_{k=1}^s h_k \I^{J_k},
			\end{equation}
			where $\mathcal M$ is a diagonal dominated symmetric M-biquadratic tensor, $s$ is a nonnegative integer, $h_k>0$, $J_k=\{(i_k,j_k)\}$,  $i_k\in[m]$, $j_k\in[n]$, and $J_s\subset J_{s-1}\subset \cdots \subset J_1$.
			If $\B$ is a symmetric B-quadratic tensor. Then either $\B$ is a strictly diagonal dominated symmetric M-quadratic tensor itself, or it can be decomposed as \eqref{equ_decomp} with  $\mathcal M$ being  a strictly diagonal dominated symmetric M-biquadratic tensor.
		\end{Thm}
		\begin{proof}
			For any given {symmetric} B$_0$-quadratic tensor $\B=(b_{i_1j_1i_2j_2})\in SBQ(m,n)$, define
			\begin{equation}\label{J1}
				J_1=\left\{(i,j)\in[m]\otimes [n]:  \exists (i_2,j_2) \neq (i,j) \text{ such that } b_{iji_2j_2} >0 \right\}.
			\end{equation}
			If $J_1=\emptyset$, then $\B$ itself is already a Z-biquadratic tensor. It follows from Proposition~\ref{prop_Z-B-tensor} that  $\B$ is  a diagonally dominated symmetric Z-biquadradtic tensor.
			{Furthermore, it follows from~\eqref{sym_Btensor} that
				\begin{equation*}
					b_{ijij} \ge mn \max_{i_2j_2} b_{iji_2j_2} -  \sum_{i_2=1}^m \sum_{j_2=1}^n  \bar b_{iji_2j_2} \ge 0.
				\end{equation*}
				Namely, all diagonal elements of $\B$ are nonnegative.}
			By virtue of Corollary~\ref{Cor_M-B-tensor}, $\B$ is a symmetric M-biquadratic tensor itself.
			
			If $J_1\neq \emptyset$, denote $\B_1:=\B$,
			\begin{equation*}
				d_{ij}:=\max_{(i_2,j_2)\neq (i,j)} b_{iji_2j_2}, \ \forall (i,j)\in J_1 \text{ and } h_1=\min_{(i,j)\in J_1} d_{ij}.
			\end{equation*}
			By \eqref{J1},  we have $d_{ij}>0$ and $h_1>0$.  Let $\B_2:=\B_1-h_1 \I^{J_1}$. Now we claim that $\B_2$ is still a $B_0$-biquadratic tensor.
			In fact, we have
			\begin{eqnarray*}
				\sum_{i_2=1}^m \sum_{j_2=1}^n  (\B_2)_{iji_2j_2}&=& 	 \sum_{i_2=1}^m \sum_{j_2=1}^n b_{iji_2j_2} -  k_{ij}h_1 \\
				&\ge & mn d_{ij} -k_{ij}h_1 \\
				&\ge& mn(d_{ij}-h_1)\\
				&\ge & 0,
			\end{eqnarray*}
			where $k_{ij}$ is the number of nonzero elements in $\I^{J_1}(i,j,:,:)$, and the first inequality follows from \eqref{sym_Btensor}.
			By the symmetry of $\B_1$, we know that if $b_{iji_2j_2}>0$ for some $(i_2,j_2)\neq (i,j)$, then $(i_2,j_2)$, $(i,j_2)$,  and $(i_2,j)$ are also in $J_1$.
			Thus, for any $(i,j)\in J_1$,
			\begin{equation*}
				\max_{(i_2,j_2)\neq (i,j)} (\B_2)_{iji_2j_2} =\max_{(i_2,j_2) \in J_1} (\B_2)_{iji_2j_2} = d_{ij}-h_1\ge 0.
			\end{equation*}
			Therefore, we have
			\begin{equation*}
				\frac{1}{mn}\sum_{i_2=1}^m \sum_{j_2=1}^n  (\B_2)_{iji_2j_2} \ge \max_{(i_2,j_2)\neq (i,j)} (\B_2)_{iji_2j_2} \ge 0.
			\end{equation*}
			Thus, we complete the proof of the claim that   $\B_2$ is still a $B_0$-biquadratic tensor.

			Continue the {above} procedure until the remaining part $\B_s$ is an M-tensor. It is not hard to find out that $J_{k+1} = J_k-\hat J_k$ with $\hat J_k=\{(i,j)\in J_k: d_{ij}=h_k\}$ for any $k\in\{1,\dots,s-1\}$.  Similarly we can show the second part of the theorem for
			symmetric B-biquadratic tensors.
		\end{proof}

		Based on the above theorem, we show the following result.
		\begin{Cor}
			{A  {symmetric} B$_0$-biquadratic} tensor is positive semi-definite, and {a {symmetric} B-biquadratic tensor} is positive definite.
		\end{Cor}
		\begin{proof}
			The results follow directly from Theorem~\ref{Thm:B_decomp} and the fact that the summation of  {positive}  semi-definite biquadratic tensors are  positive  semi-definite. Furthermore,  the summation of a  {positive}  definite biquadratic tensor and several positive  semi-definite biquadratic tensors are  {positive}   definite.
		\end{proof}

		
		\section{A Riemannian LBFGS Method for Computing The Smallest M-Eigenvalue of A Biquadratic Tensor}
		
		We consider the following  two-unit-sphere constrained optimization problem:
		\begin{equation}\label{equ:opt}
			\min_{\vx\in\Re^m, \vy\in \Re^n}\quad \frac{\A \vx\vy\vx\vy}{\vx^\top \vx \vy^\top \vy} \quad \text{s.t.} \quad \vx^\top \vx=1, \ \vy^\top \vy = 1.
		\end{equation}	
		Since the linearly independent constraint qualification is satisfied, every local optimal solution must satisfy the KKT conditions:
		\begin{eqnarray}
			\nonumber &&	\A \cdot \vy\vx\vy+\A \vx \vy\cdot\vy - 2(\A \vx\vy\vx\vy) \vx =2\mu_1 \vx, \\
			&&	\A\vx  \cdot \vx\vy+\A \vx \vy\vx \cdot - 2(\A \vx\vy\vx\vy)\vy=2\mu_2 \vy,  \label{equ:KKT}\\
			\nonumber	&&\vx^\top \vx=1, \ \vy^\top \vy = 1.
		\end{eqnarray}
		Here, we have simplified the above equations by  taking into account the constraints $\vx^\top \vx=1$ and $\vy^\top \vy = 1$.
		By multiplying $\vx^\top$ and $\vy^\top$ at both hand sides of the first two  equations of \eqref{equ:KKT}, we have
		\[\mu_1=\mu_2=0.\]
		Thus,  \eqref{equ:KKT} is equivalent to the definitions of M-eigenvalues in \eqref{e5}, \eqref{e6}, \eqref{e7} with $\lambda=\A \vx\vy\vx\vy$.
		In other words, every KKT point of \eqref{equ:opt} corresponds to an M-eigenvector of $\A$ with the associated M-eigenvalue given by  $\lambda=\A \vx\vy\vx\vy$.
		Moreover, the smallest M-eigenvalue and its corresponding M-eigenvectors of $\A$ are associated with the global optimal solution of \eqref{equ:opt}. As established in Theorem~\ref{T3.1}, the smallest M-eigenvalue can be utilized to verify the positive semidefiniteness (definiteness) of $\A$.
		

		For convenience, let $\vz=[\vx^\top,\vy^\top]^\top\in\Re^{m+n}$ and denote  \[f(\vz) = \frac{\A \vx\vy\vx\vy}{\vx^\top \vx \vy^\top \vy}.\]
		Then the gradient of $f(\vz)$ is given by
		\begin{eqnarray*}
			&&\nabla f(\vz) =
			\begin{bmatrix}
				\frac{\A \cdot \vy\vx\vy+\A \vx \vy\cdot\vy}{(\vx^\top \vx) (\vy^\top \vy)} - 2\frac{(\A \vx\vy\vx\vy) (\vy^\top \vy) \vx}{(\vx^\top \vx)^2 (\vy^\top \vy)^2} \\
				\frac{\A\vx  \cdot \vx\vy+\A \vx \vy\vx \cdot}{(\vx^\top \vx) (\vy^\top \vy)} - 2\frac{(\A \vx\vy\vx\vy )(\vx^\top \vx)\vy}{(\vx^\top \vx)^2 (\vy^\top \vy)^2}
			\end{bmatrix}.
		\end{eqnarray*}
		Under the constraints $\vx^\top \vx=1$ and $\vy^\top \vy = 1$, the gradient    simplifies  to
		\begin{eqnarray}\label{opt_grad}
			&&\nabla f(\vz) =
			\begin{bmatrix}
				\A \cdot \vy\vx\vy+\A \vx \vy\cdot\vy - 2(\A \vx\vy\vx\vy)  \vx \\
				\A\vx  \cdot \vx\vy+\A \vx \vy\vx \cdot - 2(\A \vx\vy\vx\vy )\vy
			\end{bmatrix}.
		\end{eqnarray}
		Therefore, the definitions of M-eigenvalues in \eqref{e5}, \eqref{e6}, \eqref{e7} is also equivalent to
		\[\nabla f(\vz)=\0, \  \vx^\top \vx=1, \ \vy^\top \vy = 1.\]

		Furthermore, it can be validated that
		\[\nabla_{\vx} f(\vz)^\top \vx = 0 \text{ and } \nabla_{\vy} f(\vz)^\top \vy = 0.\]
		Consequently, the partial gradients  $\nabla_{\vx} f(\vz)$ and $\nabla_{\vy} f(\vz)$  inherently reside in the tangent spaces of their respective unit spheres. This geometric property provides fundamental motivation for our selection of the optimization model \eqref{equ:opt}, as it naturally aligns with the underlying manifold structure of the problem.

		Next, we present a Riemannian LBFGS  (limited memory BFGS)  method for solving \eqref{equ:opt}.
		At the $k$-th step, the LBFGS method generate a search direction as
		\[\vp^{(k)} = -H^{(k)}   \nabla f(\vz^{(k)}),\]
		where $H^{(k)}\in \Re^{(m+n)\times (m+n)}$ is the quasi-Newton matrix. Notably, the L-BFGS method approximates the BFGS method without explicitly storing the dense Hessian matrix $H^{(k)}$. Instead, it uses a limited history of updates to implicitly represent the Hessian or its inverse. This  could be done in  the two-loop recursion \cite{NW06}.
		To establish theoretical convergence guarantees, we impose the following essential conditions:
		\begin{equation}\label{equ:descent}
			\left\{
			\begin{array}{l}	
				\left(\vp_{\vx}^{(k)} \right)^\top \nabla_{\vx} f(\vz^{(k)}) \le - C_L \left\| \nabla_{\vx} f(\vz^{(k)}) \right\|^2,\\
				\left(\vp_{\vy}^{(k)} \right)^\top \nabla_{\vy} f(\vz^{(k)}) \le -C_L \left\| \nabla_{\vy} f(\vz^{(k)}) \right\|^2,
			\end{array}
			\right.
		\end{equation}
		and
		\begin{equation}\label{equ:P_upbd}
			\left\{
			\begin{array}{l}	
				\|\vp_{\vx}^{(k)}\|\le C_U \left\| \nabla_{\vx} f(\vz^{(k)}) \right\|,\\
				\|\vp_{\vy}^{(k)} \|  \le C_U\left\| \nabla_{\vy} f(\vz^{(k)}) \right\|.
			\end{array}
			\right.
		\end{equation}
		Here, $C_L\le 1\le C_U$ are predefined parameters.
		These conditions may not always hold for the classic LBFGS method.  Consequently, when the descent condition  \eqref{equ:descent} or \eqref{equ:P_upbd} is not satisfied, we adopt a safeguard strategy by setting the  search direction as  $\vp^{(k)} = -   \nabla f(\vz^{(k)})$. Namely,
		\begin{equation}\label{equ:pk}
			\vp^{(k)} = 	\left\{\begin{array}{cc}
				-H^{(k)}   \nabla f(\vz^{(k)}), & \text{ if } \eqref{equ:descent}, \eqref{equ:P_upbd} \text{ hold},  \\
				-  \nabla f(\vz^{(k)}), & \text{ otherwise.}
			\end{array}\right.
		\end{equation}

		The classic LBFGS method was originally developed for unconstrained  optimization problems. In contrast, our optimization model \eqref{equ:opt} involves additional two unit sphere constraints. To address this challenge while maintaining the convergence properties of LBFGS, we employ a modified approach inspired by Chang et al. \cite{CCQ16}. Specifically, we incorporate the following Cayley transform, which inherently preserves the spherical constraints:
		\begin{eqnarray}
			\label{x_update}\vx^{(k+1)}(\alpha) = \frac{\left[\left(1-\alpha (\vx^{(k)})^\top \vp_{\vx}^{(k)}\right)^2 - \alpha^2 \left\| \vp_{\vx}^{(k)}\right\| \right] \vx^{(k)}+2\alpha  \vp_{\vx}^{(k)}}{1+ \alpha^2 \left\| \vp_{\vx}^{(k)}\right\|^2-\left(\alpha (\vx^{(k)})^\top \vp_{\vx}^{(k)}\right)^2}, \\
			\label{y_update}	\vy^{(k+1)}(\alpha) = \frac{\left[\left(1-\alpha (\vy^{(k)})^\top \vp_{\vy}^{(k)}\right)^2 - \alpha^2 \left\| \vp_{\vy}^{(k)}\right\| \right] \vy^{(k)}+2\alpha  \vp_{\vy}^{(k)}}{1+ \alpha^2 \left\| \vp_{\vy}^{(k)}\right\|^2-\left(\alpha (\vy^{(k)})^\top \vp_{\vy}^{(k)}\right)^2 }.
		\end{eqnarray}
		Here, $\alpha >0$ represents the step size, which is determined through a  backtracking line search procedure to ensure the Armijo condition is satisfied, i.e.,
		\begin{equation}\label{linesearch}
			f(\vz^{(k+1)}(\alpha^{(k)})) \le  f(\vz^{(k)}) + \eta \alpha^{(k)} \left(\vp^{(k)} \right)^\top \nabla f(\vz^{(k)}),
		\end{equation}
		where $\eta \in (0,2)$ a predetermined constant that controls the descent amount.
		The stepsize $\alpha^{(k)}$ is determined through an iterative reduction process, starting from an initial value $\alpha =1$ and  gradually decrease by $\alpha = \beta \alpha$   until the condition specified in equation \eqref{linesearch} is satisfied. Here, $\beta <1$ is the   decrease ratio.
		
		The Riemannian  LBFGS algorithm terminates when the following stopping condition is satisfied:
		\begin{equation}\label{equ:stop}
			\frac{\|\vz^{(k+1)} - \vz^{(k)}\|}{\|\vz^{(k)}\|} \le \epsilon_1,  \ \|\nabla  f(\vz^{(k+1)})\|\le \epsilon_2, \ \frac{f(\vz^{(k+1)}) - f(\vz^{(k)})}{ f(\vz^{(k)})+1} \le \epsilon_3,
		\end{equation}
		where $\epsilon_1, \epsilon_2, \epsilon_3$ are small positive tolerance parameters approaching zero.
		
		We present the  Riemannian LBFGS method for computing M-eigenvalues of biquadratic tensors in Algorithm~\ref{Alg:R-LBFGS}.
		
		\begin{algorithm}[t]
			\caption{A Riemannian LBFGS method for computing M-eigenvalues of biquadratic tensors} \label{Alg:R-LBFGS}
			\begin{algorithmic} [1]
				\Require $\A\in BQ(m,n)$, $\vx^{(0)}\in \Re^m$, $\vy^{(0)}\in \Re^n$,  $\eta \in (0,2)$, $k_{\max}$,  $\epsilon_1, \epsilon_2, \epsilon_3 \in (0,1)$, $C_L\le 1 \le C_U$.
				\For{$k=0,\dots,k_{\max},$}
				\State Compute   $\nabla f(\vz^{k})$  using   \eqref{opt_grad} and then determine the descent direction $\vp^{(k)}$ by  \eqref{equ:pk}.
				\State Compute the stepsize $\alpha^{(k)}$ through a backtracking   procedure  until  \eqref{linesearch} is satisfied.
				\State Update $\vx^{(k+1)}$ and $\vy^{(k+1)}$ by \eqref{x_update} and \eqref{y_update},  respectively.
				\If{The stopping criterion \eqref{equ:stop} is reached}
				\State Stop.
				\EndIf
				\EndFor
				\State \textbf{Output:} $\vx^{(k+1)}$ and $\vy^{(k+1)}$
			\end{algorithmic}
		\end{algorithm}
		
		\subsection{Convergence analysis}
		We now establish the convergence analysis of Algorithm~\ref{Alg:R-LBFGS}. Our Riemannian LBFGS method is specifically designed for solving the optimization problem~\eqref{equ:opt} that is constrained on  the product of two unit spheres. This setting differs   from existing Riemannian LBFGS methods that primarily address optimization problems constrained to a single unit sphere  \cite{CCQ16}, as well as classic LBFGS methods developed for unconstrained optimization problems \cite{NW06}.

		We begin with several lemmas.
		
		\begin{Lem}\label{Lem:Pk}
			Let $\vz^{(k)}$ be a sequence generated by Algorithm~\ref{Alg:R-LBFGS}. Consider the descent direction $\vp^{(k)}$ defined by \eqref{equ:pk}.
			For any  positive constants $C_L\le 1$ and $C_U\ge 1$, the conditions
			\eqref{equ:descent} and \eqref{equ:P_upbd} hold.
		\end{Lem}
		\begin{proof}
			Consider the two choices of  the descent direction $\vp^{(k)}$  as defined in  \eqref{equ:pk}.
			
			If $\vp^{(k)} = -\nabla f(\vz^{(k)})$,   the conclusions of this lemma   follow directly by setting $C_L=1$ and $C_U=1$.
			
			If the descent direction $\vp^{(k)}$ is selected according to the L-BFGS method,  then both the  conditions  \eqref{equ:descent}  and \eqref{equ:P_upbd} must be satisfied.
			
			Combining the both cases, the proof is completed.
		\end{proof}

		\begin{Lem}\label{Lem:grad_bound}
			The objective function  in \eqref{equ:opt},  its gradient vector and   Hessian matrix, are  bounded on any feasible set of \eqref{equ:opt}. Namely, there exists a positive constant $M$ such that
			\begin{equation}\label{eq:grad_bound}
				\left\|f(\vz)\right\|\le M, \	\left\|\nabla f(\vz)\right\|\le M, \  \left\|\nabla^2 f(\vz)\right\|\le M,
			\end{equation}
			for any $\vz=(\vx^\top, \vy^\top)^\top$ satisfying $\vx^\top \vx=1$ and  $\vy^\top \vy=1$.
		\end{Lem}
		\begin{proof}
			Under the conditions that  $\vx^\top\vx=1$ and  $\vy^\top \vy=1$,   $f(\vz)$, $\nabla f(\vz)$ and $\nabla^2 f(\vz)$ are all polynomial functions over $\vz$. Then it follows from   $\|\vz\|=\sqrt{2}$ is a compact set such that \eqref{eq:grad_bound} holds.
		\end{proof}

		\begin{Lem}\label{Lem:stepsize_bound}
			Let $\vz^{(k)}$ be a sequence generated by Algorithm~\ref{Alg:R-LBFGS}.   Let $\hat{\alpha} \equiv \frac{(2-\eta)C_L}{(2+\eta)MC_U^2}$ and   $\alpha_{\min}=\beta \hat{\alpha}$, where $\beta<1$ is the   decrease ratio in the  Armijo line search.
			Then for all $k\ge 0$, the  stepsize $\alpha^{(k)}$ that satisfies ~\eqref{linesearch} is lower bounded, i.e.,
			\[\alpha_{\min} \le \alpha^{(k)}\le 1.\]
		\end{Lem}
		\begin{proof}
			By \eqref{eq:grad_bound}, we have
			\begin{eqnarray*}
				&&f(\vz^{(k+1)}(\alpha)) - f(\vz^{(k)}) \\
				&\le & \nabla_{\vx} f(\vz^{(k)})^\top  \left(\vx^{(k+1)}(\alpha) -\vx^{(k)}\right) +\frac{M}2 \left\|\vx^{(k+1)}(\alpha) -\vx^{(k)}\right\| ^2\\
				&&+ \nabla_{\vy} f(\vz^{(k)})^\top  \left(\vy^{(k+1)}(\alpha) -\vy^{(k)}\right) +\frac{M}2 \left\|\vy^{(k+1)}(\alpha) -\vy^{(k)}\right\| ^2.
			\end{eqnarray*}
			
			Following from Lemma 4.3 in \cite{CCQ16},
			for any $\alpha$ satisfying $0<\alpha \le \hat{\alpha}$, it holds that
			\begin{eqnarray*}
				&&\nabla_{\vx} f(\vz^{(k)})^\top  \left(\vx^{(k+1)}(\alpha) -\vx^{(k)}\right) +\frac{M}2 \left\|\vx^{(k+1)}(\alpha) -\vx^{(k)}\right\| ^2\\
				&\le& \eta \alpha 	\left(\vp_{\vx}^{(k)} \right)^\top \nabla_{\vx} f(\vz^{(k)})
			\end{eqnarray*}
			and
			\begin{eqnarray*}
				&&\nabla_{\vy} f(\vz^{(k)})^\top  \left(\vy^{(k+1)}(\alpha) -\vy^{(k)}\right) +\frac{M}2 \left\|\vy^{(k+1)}(\alpha) -\vy^{(k)}\right\| ^2\\
				&\le& \eta \alpha 	\left(\vp_{\vy}^{(k)} \right)^\top \nabla_{\vy} f(\vz^{(k)}).
			\end{eqnarray*}
			Therefore, we have
			\[f(\vz^{(k+1)}(\alpha)) - f(\vz^{(k)}) \le \eta \alpha 	\left(\vp^{(k)} \right)^\top \nabla f(\vz^{(k)})\]
			for any  $0<\alpha \le \hat{\alpha}$.
			By the backtracking line search rule, we have  $\alpha^{(k)}\ge \beta \hat{\alpha}$ for all $k$.  This completes the proof.
		\end{proof}

		We now prove that the objective function in \eqref{equ:opt} converges to a constant value, and any limit point is   a KKT point.
		\begin{Thm}\label{Thm:grad_goto_zero}
			Let $\vz^{(k)}$ be a sequence generated by Algorithm~\ref{Alg:R-LBFGS}.    Then we have
			\begin{equation}\label{eq:des_amount}
				f(\vz^{(k+1)}) - f(\vz^{(k)}) \le -\eta \alpha_{\min} C_L  \left\|\nabla f(\vz^{(k)})\right\|^2.
			\end{equation}
			Furthermore, the objective function  $f(\vz^{(k)})$ monotonously converges to a constant value and
			\begin{equation}
				\lim_{k\rightarrow +\infty} \left\|\nabla f(\vz^{(k)})\right\| =0.
			\end{equation}
			Namely, any {limit} point  of  Algorithm~\ref{Alg:R-LBFGS} is a KKT point of problem  \eqref{equ:opt}, which is also an M-eigenpair of the biquadratic tensor $\A$.
			
		\end{Thm}
		\begin{proof}
			It follows from Lemma~\ref{Lem:stepsize_bound} that
			\[f(\vz^{(k+1)}) - f(\vz^{(k)}) \le \eta \alpha_{\min} 	\left(\vp^{(k)} \right)^\top \nabla f(\vz^{(k)}).\]
			Combining  with Lemma~\ref{Lem:Pk} derives \eqref{eq:des_amount}.
			In other words, the objective function  $f(\vz^{(k)})$ is  monotonously decreasing.  Lemma~\ref{Lem:grad_bound} shows that $f(\vz)$
			is lower bounded. Therefore, $f(\vz^{(k)})$ monotonously converges to a constant value $\lambda_{*}$.

			Furthermore,  summing the both hand sides of \eqref{eq:des_amount} over $k$ gives
			\begin{eqnarray*}
				\lambda_{*}-f(\vz^{0})	&=&\sum_{k=0}^{\infty} f(\vz^{(k+1)}) - f(\vz^{(k)}) \le -\eta \alpha_{\min} C_L  \sum_{k=0}^{\infty}  \left\|\nabla f(\vz^{(k)})\right\|^2.
			\end{eqnarray*}
			Therefore, we have
			\begin{eqnarray*}
				\sum_{k=0}^{\infty}  \left\|\nabla f(\vz^{(k)})\right\|^2 \le \frac{\lambda_{*}-f(\vz^{0})}{\eta \alpha_{\min} C_L},
			\end{eqnarray*}
			which implies $\lim_{k\rightarrow +\infty} \left\|\nabla f(\vz^{(k)})\right\| =0$.
			
			This completes the proof. 		
		\end{proof}

		Before demonstrating the  sequential global convergence, we first present the following lemma.
		
		\begin{Lem}\label{Lem:grad_xmx_rel}
			Let $\vz^{(k)}$ be a sequence generated by Algorithm~\ref{Alg:R-LBFGS}.    Then there exist  two positive constants $l\le u$ such that
			\begin{equation}\label{equ:nabla_z-z_bound}
				l\left\|\vz^{(k+1)} - \vz^{(k)}\right\| \le \left\|\nabla f(\vz^{(k)})\right\| \le u\left\|\vz^{(k+1)} - \vz^{(k)}\right\|.
			\end{equation}
			Here, $l = \frac{1}{2C_U}$ and $u=\frac{C_U(1+C_UM)}{2\alpha_{\min}C_L^2}$.
		\end{Lem}
		\begin{proof}
			By Theorem~4.6 in \cite{CCQ16}, we have
			\begin{equation*}
				\left\{\begin{array}{c}
					\left\|\vx^{(k+1)} - \vx^{(k)}\right\| \le 2C_U\alpha^{(k)} \left\|\nabla_{\vx} f(\vz^{(k)})\right\|,\\
					\left\|\vy^{(k+1)} - \vy^{(k)}\right\| \le 2C_U\alpha^{(k)} \left\|\nabla_{\vy} f(\vz^{(k)})\right\|.
				\end{array}\right.
			\end{equation*}
			Therefore, we have
			\[\left\|\vz^{(k+1)} - \vz^{(k)}\right\| \le 2C_U\alpha^{(k)} \left\|\nabla  f(\vz^{(k)})\right\|\le 2C_U\left\|\nabla  f(\vz^{(k)})\right\|.\]
			Here, the last inequality follows from $\alpha^{(k)}\le 1$.
			This shows the left inequality in \eqref{equ:nabla_z-z_bound} holds.
			
			Furthermore, by Lemma~4.7 in \cite{CCQ16}, we have
			\begin{equation*}
				\left\{\begin{array}{c}
					\left\|\vx^{(k+1)} - \vx^{(k)}\right\| \ge \frac{2\alpha_{\min}C_L}{C_U(1+C_UM)} \left\|\vp_{\vx}^{(k)}\right\|\ge   \frac{2\alpha_{\min}C_L^2}{C_U(1+C_UM)} \left\|\nabla_{\vx} f(\vz^{(k)})\right\|,\\
					\left\|\vy^{(k+1)} - \vy^{(k)}\right\| \ge \frac{2\alpha_{\min}C_L}{C_U(1+C_UM)} \left\|\vp_{\vy}^{(k)}\right\|\ge   \frac{2\alpha_{\min}C_L^2}{C_U(1+C_UM)} \left\|\nabla_{\vy} f(\vz^{(k)})\right\|.
				\end{array}\right.
			\end{equation*}
			Therefore, we have
			\[\left\|\vz^{(k+1)} - \vz^{(k)}\right\| \ge \frac{2\alpha_{\min}C_L^2}{C_U(1+C_UM)} \left\|\nabla f(\vz^{(k)})\right\|.\]
			
			This completes the proof.
		\end{proof}
		
		Now we are ready to present   the global convergence of the sequence produced  by Algorithm~\ref{Alg:R-LBFGS}.
		\begin{Thm}\label{Thm:seq_conv}
			Let $\vz^{(k)}$ be a sequence generated by Algorithm~\ref{Alg:R-LBFGS}.  Then we have the following result,
			\begin{equation}\label{equ:Cauchy}
				\sum_{k=0}^\infty \left\|\vz^{(k+1)} - \vz^{(k)}\right\| \le +\infty.
			\end{equation}
			Namely, the sequence $\vz^{(k)}$ converges  globally to a limit point $\vz^{*}$. Furthermore,  $\vz^{*}$ is an M-eigenvector of $\A$ with the corresponding eigenvalue $\lambda^* = \A \vx^{*}\vy^{*}\vx^{*}\vy^{*}$.
		\end{Thm}
		\begin{proof}
			By Theorem~\ref{Thm:grad_goto_zero} and Lemma~\ref{Lem:grad_xmx_rel}, we may deduce the sufficient decrease property
			\[	f(\vz^{(k+1)}) - f(\vz^{(k)}) \le -\eta \alpha_{\min} C_L  l^2\left\|\vz^{(k+1)} - \vz^{(k)}\right\|^2,\]
			and the gradient lower bound for the iterates gap
			\[\left\|\nabla f(\vz^{(k)})\right\| \le \frac{1}{l}\left\|\vz^{(k+1)} - \vz^{(k)}\right\|.\]
			
			Both the objective function and the constraints in \eqref{equ:opt} are semi-algebraic functions that satisfy the KL property \cite{BST14}.
			Therefore, it follows from Theorem 1 in \cite{BST14} that \eqref{equ:Cauchy} holds.
			In other words, $\vz^{(k)}$  is a Cauchy sequence, and it  converges   globally to a limit point $\vz^{*}$.
			
			This completes the proof.
		\end{proof}

		\section{Numerical Experiments}

		\subsection{Inclusion Intervals of M-Eigenvalues}
		In this subsection, we show the bounds of M-eigenvalues presented in  Theorem~\ref{Ger} by several   examples.
		{We begin with the elastic moduli tensor taken from Example 4.2 in Zhao \cite{Zh23}.}
		
		\begin{example}\label{ex:elastic_tensor}
			Consider the elastic moduli tensor $\C=(c_{i_1j_1i_2j_2})\in BQ(2,2)$  for the Tetragonal system  with 7 elasticities  in the equilibrium equation~\eqref{eq:elasticity_tensor}, where
			\begin{eqnarray*}
				&&c_{1123} = c_{1131} = c_{2223} = c_{2231} = c_{3323} = c_{3331} = c_{3312} = c_{2331} = c_{2312} = c_{3112} = 0,\\
				&&c_{1112} = -c_{2212}, c_{2222} = c_{1111}, c_{2233} = c_{1133}, c_{3131} = c_{2323},
			\end{eqnarray*}
			with
			\[c_{1111} = 4, c_{1122} = -4, c_{1133} = -2, c_{1112} = 1, c_{3333} = 3, c_{2323} = 4, c_{1212} = 4.\]
			We then transform $\C$ into a  symmetric biquadratic  tensor $\A$
			{by $a_{i_1j_1i_2j_2} = \frac14 \left(a_{i_1j_1i_2j_2}+a_{i_2j_1i_1j_2}+a_{i_1j_2i_2j_1}+a_{i_2j_2i_1j_1}\right)$ and show the results as follows.}		
			\begin{eqnarray*}
				&&a_{1111} = 4, a_{1112} = a_{1211} = 1, a_{1121} = a_{2111} = 1, a_{1212} = 4, a_{1222} = a_{2212} = -1,\\
				&&a_{1133} = a_{1331} = a_{3113} = a_{3311} = 1, a_{1313} = 4, a_{2121} = 4, a_{2122} = a_{2221} = -1,\\
				&&a_{2222} = 4, a_{2233} = a_{2332} = a_{3223} = a_{3322} = 1, a_{2323} = 4, a_{3131} = 4, a_{3232} = 4,\\
				&&a_{3333} = 3.
			\end{eqnarray*}
			
		\end{example}
		
		The lower and upper bounds for the M-eigenvalues, as  computed by  He, Li and Wei \cite{HLW20}, Zhao \cite{Zh23}, and Theorem~\ref{Ger}, are shown in Table~\ref{tab:ex7.1}.
		Notably, since $\C$ is nonsymmetric, the results of \cite{HLW20} and  \cite{Zh23} are not applicable here.
		We also keep most entries of $\A$ unchanged and adjusted the values of $a_{1212}$,  $a_{1312}$ and $a_{1213}$, and $a_{1313}$ respectively to generate additional examples. From this table, we see that Theorem~\ref{Ger} could  generate  smaller intervals compared with other methods. In particular, for the instances with $a_{1212}=2$ and $a_{1312}=a_{1213}=2$,  Theorem~\ref{Ger} could show positive semi-definiteness, and for the instance with  $a_{1313}=2$,   Theorem~\ref{Ger} could show positive  definiteness, while the other methods may not obtain the same conclusion.
		
		\begin{table}
			\begin{center}
				\caption{Interval of M-eigenvalues in Example 7.1.   Here, `$-$' denotes the corresponding algorithm is  not applicable. Numbers in bold are relatively better.}\label{tab:ex7.1}
				\begin{tabular}{c|c|c|c}
					\hline
					& He, Li and Wei \cite{HLW20} 	& Zhao \cite{Zh23} & Theorem~\ref{Ger}\\ \hline
					$\C$ & $-$&$-$& {\bf [$-$5,13]}	\\
					$\A$ &  [$-$1,9] &  {\bf [1,7]} & {\bf [1,7]} \\
					$\A$ with $a_{1212}=2$&   [$-$3, 9] & [$-$1,7] & {\bf [0,7]}\\
					$\A$ with $a_{1312}=a_{1213}=2$&   [$-$3, 11] & [$-$0.8637,    8.8637] & {\bf [0,8]}\\
					$\A$ with $a_{1313}=2$&   [$-$2, 9] & [0,7] &{\bf [1,7]}\\
					\hline
				\end{tabular}
			\end{center}
		\end{table}

		\begin{table}
			\begin{center}
				\caption{Interval of M-eigenvalues in Example 7.2.  Numbers in bold are relatively better.}\label{tab:ex7.2}
				\begin{tabular}{c|c|c|c|c}
					\hline
					$m$ &	$n$& He, Li and Wei \cite{HLW20} 	& Zhao \cite{Zh23} & Theorem~\ref{Ger}\\ \hline
					\multirow{4}*{2} & 2 &  [$-$8.10, 21.36] &    {\bf [$-$2.37,15.65]}   & [$-$2.68,16.07] \\
					& 5 & [$-$60.73, 82.29] & [$-$34.43, 56.00] & {\bf [$-$33.60, 55.79]} \\
					& 10 & [$-$224.41, 261.26] & [$-$135.97, 172.98] & {\bf [$-$131.85, 168.69] }\\
					& 20 & [$-$827.48, 893.86] & [$-$518.89, 584.25] & {\bf [$-$496.26, 561.33] }\\
					\hline
					
					\multirow{4}*{5} & 2 &  [$-$60.75, 82.43] &    {[$-$34.32, 56.14]}   & {\bf[ -33.52, 55.84]} \\
					& 5 & [$-$598.99, 630.01] & [$-$174.55, 205.50] & {\bf [$-$143.27, 174.20] } \\
					& 10 & [$-$1821.69, 1867.11] & [$-$498.30, 544.05] & {\bf [$-$435.28, 481.81] }\\
					& 20 & [$-$5942.73, 6018.45] & [$-$1584.09, 1660.27] & {\bf [$-$1423.54, 1501.01]} \\
					\hline
				\end{tabular}
			\end{center}
		\end{table}
		
		We continue with larger scale examples.
		
		\begin{example}
			Suppose that $\A=\B+ (m+n)\I\in SBQ(m,n)$, where $\B$ is a  random symmetric biquadratic tensor whose entries are uniformly distributed random integers ranging from 1 to $m+n$, and $\I$ denotes the identity {biquadratic} tensor.
		\end{example}
		
		We compare the lower and upper bounds of M-eigenvalues obtained by   He, Li and Wei \cite{HLW20},  Zhao \cite{Zh23}, and Theorem~\ref{Ger}.  We repeat all experiments 100 times with different random tensors $\B$ and show the average results in Table~\ref{tab:ex7.2}.  It follows from this table that Theorem~\ref{Ger} could return a { smaller} interval  on average {in most cases}.

		We also compare the quality of the intervals obtained by Zhao \cite{Zh23} and Theorem~\ref{Ger} in the following way. Denote $l_z$ and $u_z$ as the  lower and upper bounds of M-eigenvalues obtained by Zhao \cite{Zh23}, and  $l$ and $u$ as the  lower and upper bounds of M-eigenvalues obtained by Theorem~\ref{Ger}, respectively.  We continue to show the	the number cases that $l_z =, >, < l$ and  $u_z =, >, <u$, respectively, in Table~\ref{tab:ex7.2_2}.  	It follows from Table~\ref{tab:ex7.2_2} that Theorem~\ref{Ger} could return a smaller interval in  most cases. This phenomenon becomes more obvious when $m$ and $n$ are large.

		\begin{table}
			\begin{center}
				\caption{A summary of the number of instances for the relationships between the lower and upper bounds obtained by Zhao \cite{Zh23} and   Theorem 4.1, as illustrated  in Example 7.2.}\label{tab:ex7.2_2}
				\begin{tabular}{c|c|ccc|ccc}
					\hline
					$m$ & $n$ & $l_{z}=l$ & $l_z<l$ &  $l_z>l$ & $u_{z}=u$ & $u_z<u$ & $u_z>u$
					\\ \hline
					\multirow{4}*{2} & 2 &  27  &    9  &   {\bf 64}&     18  &   {\bf 75}  &    7      \\
					
					& 5&    4   & {\bf  51} &  45  &   3 &   43  &  {\bf 54}      \\
					& 10&   0   &  {\bf 78}  &  22  &   0  &  18 &   {\bf 82}      \\
					& 20 & 0  &  {\bf 95}  &   5  &   0  &   6  & {\bf 94}      \\
					\hline
					\multirow{4}*{5} & 2 &   3 &   {\bf  65}  &  32  &   3  &  45  &  {\bf 52}       \\
					&	5 &     0  &  {\bf 100} &   0  &   0 &    0   &{\bf 100}       \\
					&	10 &    0  &  {\bf 100}  &   0  &   0  &   0  & {\bf 100}      \\
					
					\hline
				\end{tabular}
			\end{center}
		\end{table}

		\subsection{Computing the Smallest M-eigenvalue}
		
		In this subsection, we present the numerical results for computing the smallest M-eigenvalues using  Algorithm~\ref{Alg:R-LBFGS}. 				For all   experiments,  {we generate   $\vx\in \Re^m$, $\vy\in \Re^n$ randomly by the  normal distribution and then normalize them to get the initial points by $\vx^{(0)} = \frac{\vx}{\|\vx\|}$ and  $\vy^{(0)} = \frac{\vy}{\|\vy\|}$, respectively.}
		To ensure robustness, each experiment is repeated  {20} times with different initial points, and the average results are reported.
		The parameters are set as follows:   $\eta=10^{-3}$, $k_{\max}=1000$,  $\epsilon_1 = 10^{-6}$, $\epsilon_2=  10^{-6}$, $\epsilon_3= 10^{-16}$, $C_L=10^{-16}$, and $C_U=10^{16}$.

		\begin{figure}
			\begin{center}
				\includegraphics[width=\linewidth]{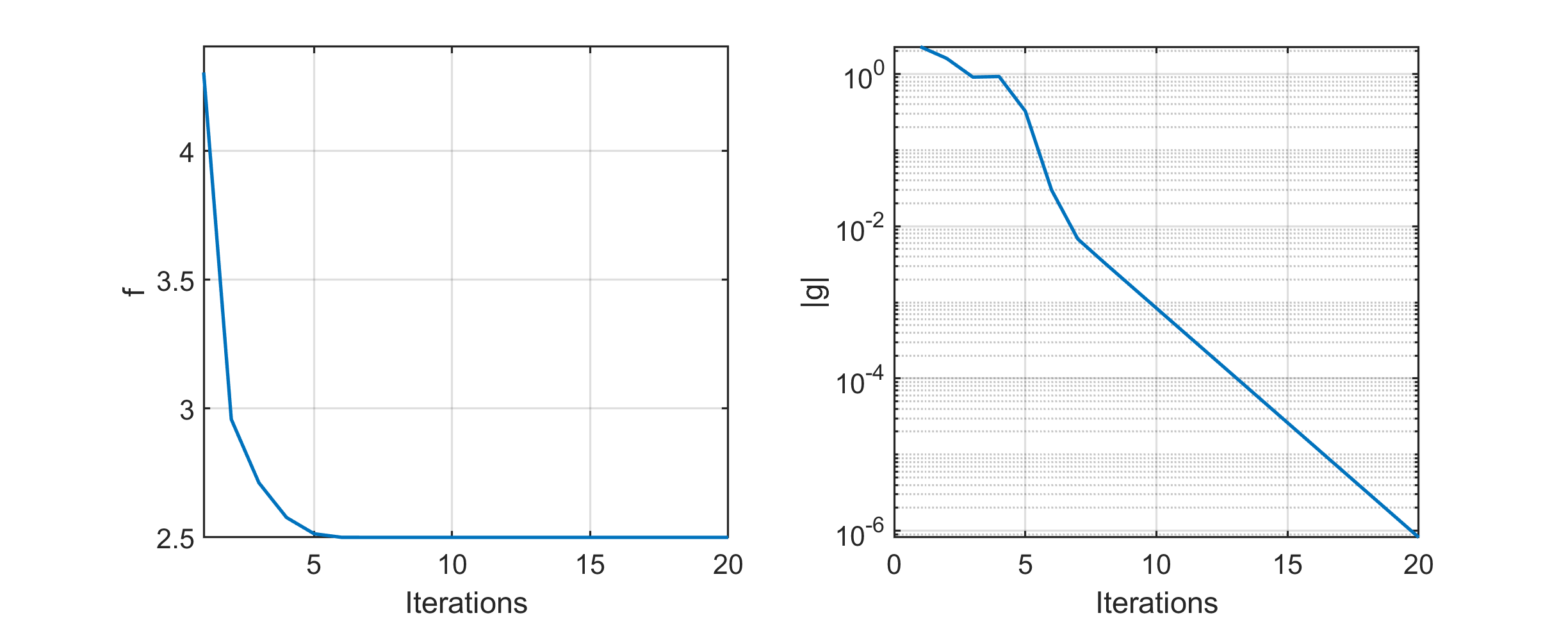}
			\end{center}
			\caption{The iterative procedure of Algorithm 1 for computing M-eigenvalues of the biquadratic tensor $\A$ in Example 7.1. The left panel shows the objective function, while the right panel displays the norm of the gradient.}\label{fig:Meig_Ex71}
		\end{figure}
		
		We first consider  the biquadratic tensor in Example~\ref{ex:elastic_tensor}.
		The average iteration count for implementing Algorithm~\ref{Alg:R-LBFGS}  is  {18.55}, the  average CPU time consumed is  $1.08\times 10^{-1}$ seconds, and the  average  value of $\|\nabla f(\vz^{\infty})\|$ is found to be $3.72\times 10^{-7}$.  {Among all 20 repeated trials, an M-eigenvalue of 2.5 was achieved in 19 cases.}  We also show the values of M-eigenpairs in Table~\ref{Tab_Meig_Ex7.1}.
		 Since for any M-eigenpair $(\lambda, \vx,\vy)$,   its variants  $(\lambda, -\vx,\vy)$, $(\lambda, \vx,-\vy)$, $(\lambda, -\vx,-\vy)$ are also M-eigenpairs of $\A$ , we  exclude these redundant M-eigenpairs from our results.
			Furthermore, we present the iterative procedure of Algorithm~\ref{Alg:R-LBFGS} for  computing M-eigenvalues in Fig.~\ref{fig:Meig_Ex71}. It is evident that our Riemannian LBFGS algorithm exhibits rapid convergence.

		\begin{table}
			\caption{M-eigenpairs computed from Example~\ref{ex:elastic_tensor}.}\label{Tab_Meig_Ex7.1}
			\begin{center}
				\begin{tabular}{c | cccccc }
					\hline
					$\lambda$ & $x_1$  & $x_2$ & $x_3$  & $y_1$  & $y_2$ & $y_3$ \\
					\hline
					2.5000 & 0.6533 & -0.2706 & 0.7071 & -0.6533 & 0.2706 & 0.7071 \\
					2.5000 & -0.6533 & 0.2706 & 0.7071 & 0.6533 & -0.2706 & 0.7071 \\
					2.5000 & 0.2706 & 0.6533 & 0.7071 & 0.2706 & 0.6533 & -0.7071 \\
					2.5000 & 0.2706 & 0.6533 & -0.7071 & 0.2706 & 0.6533 & 0.7071 \\
					3.0000 & 0.6088 & -0.7933 & 0.0003 & -0.9915 & -0.1304 & 0.0003 \\
					\hline
				\end{tabular}
			\end{center}
		\end{table}

		\begin{table}
			\begin{center}
				\caption{Numerical results for computing M-eigenvalues of the covariance biquadratic tensors in Example 7.3 by Algorithm 1.}\label{tab:ex7.3}
				\begin{tabular}{c|c|c|c|c|c|c }
					\hline
					$m$ & $n$ & $\lambda$ & Rate & Iteration & Time (s) & Res \\ 	
					\hline
					\multirow{4}*{5}        & 5 & 7.8142 & 35\% & 30.75 & 3.31e-01 & 2.15e-07 \\
					& 10 & 7.7351 & 25\% & 44.80 & 2.95e-01 & 2.22e-07 \\
					& 20 & 7.4446 &40\% & 54.90 & 4.42e-01 & 2.99e-07 \\
					& 30 & 7.3370 & 20\% & 57.35 & 5.27e-01 & 5.26e-07 \\
					\hline
					\multirow{4}*{10}
					& 5 & 7.7211 & 45\% & 43.05 & 3.19e-01 & 2.64e-07 \\
					& 10 & 7.5963 & 20\% & 47.65 & 3.52e-01 & 4.17e-07 \\
					& 20 & 7.3332 & 15\% & 64.70 & 6.43e-01 & 5.42e-07 \\
					& 30 & 7.2051 & 5\% & 80.70 & 1.10e+00 & 5.17e-07 \\
					\hline
				\end{tabular}
			\end{center}
		\end{table}
		
		At last, we present an example from statistics.
		
		\begin{example}
			We generate a sequence of 10000 independent and identically distributed  random matrices $X^{(t)}\in\Re^{m\times n}$ uniformly from the interval $[0,10]$.   Subsequently,  {we compute the covariance tensor by \eqref{eq:4cov_est}.}
		\end{example}
		
	 We show the numerical results for computing M-eigenvalues of the covariance biquadratic tensors by Algorithm~\ref{Alg:R-LBFGS} in  Table~\ref{tab:ex7.3}.
			Here,  `$\lambda$' represents the smallest M-eigenvalue obtained from 20 repeated trials, and 'Rate' denotes the success rate of returning the smallest M-eigenvalue. Additionally, 'Iteration', 'Time (s)', and 'Res' correspond to the average values of the number of iterations, the CPU time consumed, and the norm of the gradient $\|\nabla f(\vz)\|$ at the final iteration, respectively.
			
			It is confirmed that all the smallest eigenvalues obtained are positive, which aligns with Proposition~\ref{prop:psd_cov}, indicating that the covariance tensor is positive {definite}.
			The success rate remains high for small-dimensional problems but gradually decreases as the values of  $m$ and $n$ increase. This observation is consistent with the theoretical expectation that the number of M-eigenvalues for a higher-dimensional biquadratic tensor generally increases, leading to more KKT points and potentially greater difficulty in converging to the global optimal solution.
			Furthermore,  the number of iterations is consistently below 100, and the gradient norm is smaller than the specified tolerance of $10^{-6}$ in all cases.
			At last, we present the iterative procedure of Algorithm~\ref{Alg:R-LBFGS} for  computing M-eigenvalues of  the    covariance  tensor with $m=10$ and $n=30$  in Fig.~\ref{fig:Meig_Ex73}. It is evident that our Riemannian LBFGS algorithm exhibits rapid convergence.

		\begin{figure}
			\begin{center}
				\includegraphics[width=\linewidth]{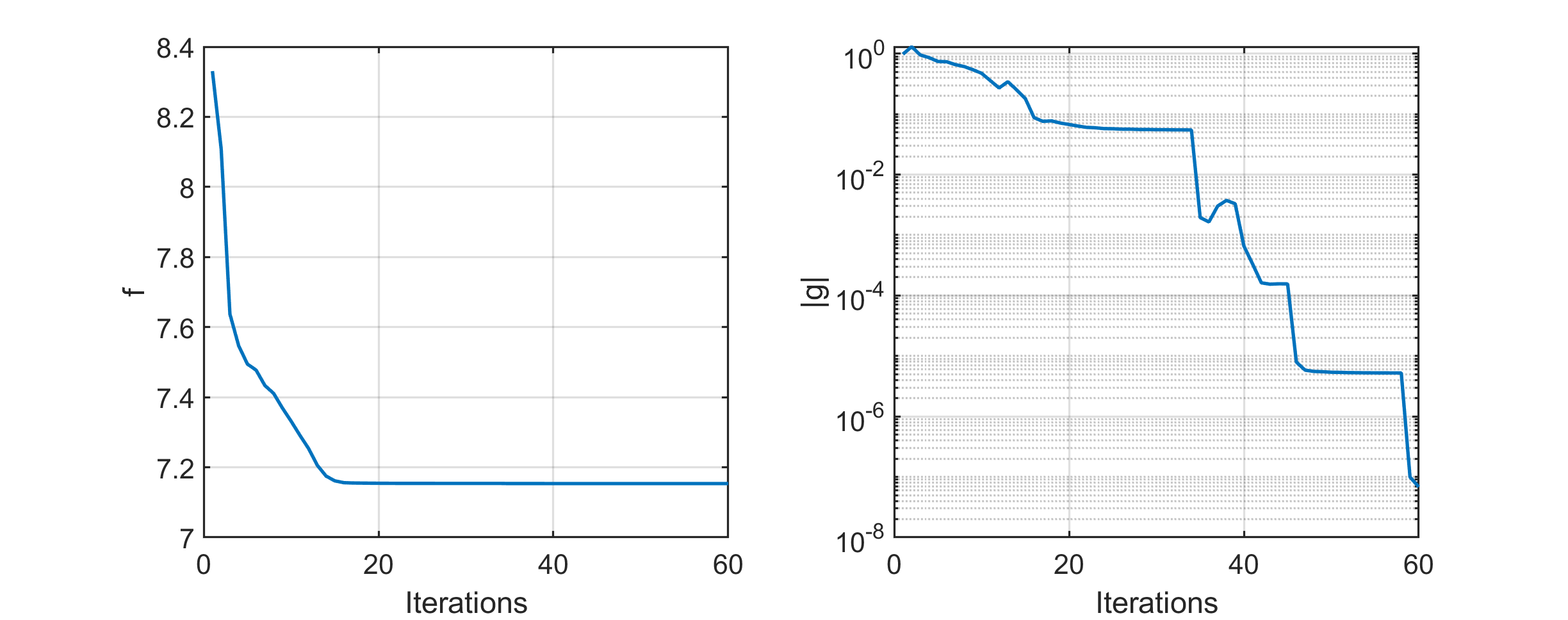}
			\end{center}
			\caption{The iterative procedure of Algorithm 1 for computing M-eigenvalues of a  nonsymmetric covariance  tensor with $m=10$ and $n=30$ in Example 7.3. The left panel shows the objective function, while the right panel displays the norm of the gradient.}\label{fig:Meig_Ex73}
		\end{figure}
		

		\section{Final Remarks}

		Motivated by applications, in this paper, we extend the definition of M-eigenvalues of symmetric biquadratic tensors to nonsymmetric biquadratic tensors, show that M-eigenvalues always exist, a biquadrtic tensor is positive semi-definite (definite) if and only if its M-eigenvalues are nonnegative (positive). We present a Gershgorin-type theorem, several structured biquadratic tensors, and an algorithm for computing the smallest M-eigenvalue of a biquadratic tensor. {We may} explore further in the following three topics.
		
		1. Study more precise bounds for M-eigenvalues of biquadratic tensors and construct more efficient algorithms for computing the smallest M-eigenvalue of a biquadratic tensor.
		
		2. Identify more classes of structured biquadratic tensors.
		
		3. Study the SOS problem of biquadratic tensors.

		\bigskip	
		
		
		{{\bf Acknowledgment}}
		This work was partially supported by Research  Center for Intelligent Operations Research, The Hong Kong Polytechnic University (4-ZZT8),    the National Natural Science Foundation of China (Nos. 12471282 and 12131004), the R\&D project of Pazhou Lab (Huangpu) (Grant no. 2023K0603), and the Fundamental Research Funds for the Central Universities (Grant No. YWF-22-T-204).
		
		

		{{\bf Data availability} Data will be made available on reasonable request.

			{\bf Conflict of interest} The authors declare no conflict of interest.}

		


	\end{document}